\documentclass[12pt]{amsart}
\usepackage{latexsym, amsmath, amssymb}
\usepackage{epsfig}
\usepackage{amscd}
\usepackage{pstricks,pst-node,cite}

\newtheorem{thm}{Theorem}[section]
\newtheorem{lem}[thm]{Lemma}

\newtheorem{cor}[thm]{Corollary}

\newtheorem{rem}[thm]{Remark}
\theoremstyle{remark}
\newcommand{\F}{\mathcal{F}}

\newcommand{\zz}{\mathbb {Z}}
\newcommand{\pol}{\hbox{ }}
\newcommand{\cp}{\nabla}

\ifx\blackandwhite\undefined
  \psset{linecolor=blue}\newrgbcolor{darkred}{.9 0 0}\newrgbcolor{emgreen}{0 .9 0}
  \newrgbcolor{lightgreen}{.7  .99 .7}\newrgbcolor{lightred}{.99 .7 .7} \else
  \psset{linecolor=blue}\newgray{darkred}{.7}\newgray{emgreen}{0.3}\newgray{pup}{.5}\newgray{xxx}{.5}
\fi
 \psset{linewidth=1pt,dash=3pt 3pt,doublesep=.05,dotsize=1pt 5}
\SpecialCoor

\begin{document}

\author{Dongseok Kim}
\address{Department of Mathematics \\Yeungnam University \\Kyongsan, 712-749, Korea} \email{dongseok@yumail.ac.kr}
\thanks{}

\author{Jaeun Lee}
\address{Department of Mathematics \\Yeungnam University \\Kyongsan, 712-749, Korea}
\email{julee@yu.ac.kr}
\thanks{The first author was supported by Korea Research Foundation Grant funded by
Korea Government (MOEHRD, Basic Research Promotion Fund)
(KRF-2006-351-C00005). The second author was supported in part by
Com$^2$MaC-KOSEF(R11-1999-054)}


\title[Some invariants of pretzel links]{Some invariants of pretzel links}
\begin{abstract}
We show that nontrivial classical pretzel knots $L(p,q,r)$ are
hyperbolic with eight exceptions which are torus knots. We find
Conway polynomials of $n$-pretzel links using a new
computation tree. As applications, we compute
the genera of $n$-pretzel links using these polynomials and find the
basket number of pretzel links by showing that the genus and
the canonical genus of a pretzel link are the same.
\end{abstract}

\maketitle

\section{Introduction}

Let $L(p_1, p_2, \ldots, p_n)$ be an \emph{$n$-pretzel link} in
$\mathbb{S}^3$ where $p_i \in \zz$ represents the number of half
twists as depicted in Figure \ref{npretzel}. In particular, if
$n=3$, it is called a \emph{classical pretzel link}, denoted by
$L(p,q,r)$. If $n$ is odd, then an $n$-pretzel link $L(p_1,p_2,
\ldots, p_n)$ is a knot if and only if none of two $p_i$'s are even,
a pretzel knot is denoted by $K(p_1,p_2, \ldots, p_n)$. If $n$ is
even, then $L(p_1,p_2, \ldots, p_n)$ is a knot if and only if only
one of the $p_i$'s is even. Generally the number of even $p_i$'s is
the number of components, unless $p_i$'s are all odd. Since pretzel
links have nice structures, they have been studied extensively. For
example, several polynomials of pretzel links have been
calculated~\cite{Hara:Qpoly, HNO:Conway, Landvoy:Jones,
Nakagawa:Alexander}. Y. Shinohara calculated the signature of
pretzel links~\cite{Shinohara:sign}. Two and three fold covering
spaces branched along pretzel knots have been
described~\cite{Bedient:double, HN:branch}. For notations and
definitions, we refer to \cite{Adams:knotbook}.

A link $L$ is \emph{almost alternating} if it is not alternating and
there is a diagram $D_L$ of $L$ such that one crossing change makes
the diagram alternating; we call $D_L$ an \emph{almost alternating
diagram}. One of the classifications of links is that they are classified
by hyperbolic, torus or satellite links~\cite{Adams:knotbook}. First
we show that classical pretzel links are prime and either alternating or
almost alternating. Menasco has shown that prime alternating knots
are either hyperbolic or torus knots~\cite{Menasco:alternating}. It
has been generalized by Adams that prime almost alternating knots
are either hyperbolic or torus knots~\cite{Adams:almost}. It is known that no
satellite knot is an almost alternating knot~\cite{Hoste}. Thus, we can classify classical pretzel knots
completely by hyperbolic or torus knots.

\begin{figure}
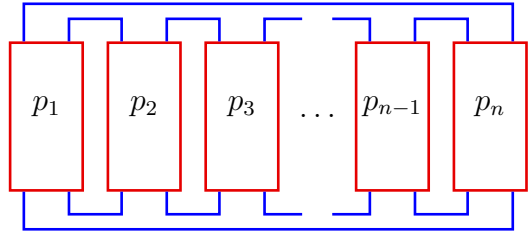

$$
\pspicture[.4](-4.3,-1.7)(3,1.7)
\psframe[linecolor=darkred](-4.1,-1)(-3.1,1) \rput[b](-3.6,0){$p_1$}
\psframe[linecolor=darkred](-2.8,-1)(-1.8,1) \rput[b](-2.3,0){$p_2$}
\psframe[linecolor=darkred](-1.5,-1)(-.5,1) \rput[b](-1,0){$p_3$}
\rput[b](0,0){$\ldots$} \psframe[linecolor=darkred](.5,-1)(1.5,1)
\rput[b](1,0){$p_{n-1}$} \psframe[linecolor=darkred](1.8,-1)(2.8,1)
\rput[b](2.3,0){$p_n$}  \psline(-3.9,1)(-3.9,1.5)(2.6,1.5)(2.6,1)
\psline(-3.3,1)(-3.3,1.3)(-2.6,1.3)(-2.6,1)
\psline(-2,1)(-2,1.3)(-1.3,1.3)(-1.3,1)
\psline(-.7,1)(-.7,1.3)(-.2,1.3) \psline(.2,1.3)(.7,1.3)(.7,1)
\psline(1.3,1)(1.3,1.3)(2,1.3)(2,1)
\psline(-3.9,-1)(-3.9,-1.5)(2.6,-1.5)(2.6,-1)
\psline(-3.3,-1)(-3.3,-1.3)(-2.6,-1.3)(-2.6,-1)
\psline(-2,-1)(-2,-1.3)(-1.3,-1.3)(-1.3,-1)
\psline(-.7,-1)(-.7,-1.3)(-.2,-1.3) \psline(.2,-1.3)(.7,-1.3)(.7,-1)
\psline(1.3,-1)(1.3,-1.3)(2,-1.3)(2,-1)
\endpspicture
$$
\caption{An $n$-pretzel link $L(p_1, p_2, \ldots , p_n)$}
\label{npretzel}
\end{figure}

Let $L$ be a link in $\mathbb{S}^3$. A compact orientable surface
$\F$ is a \emph{Seifert surface} of $L$ if the boundary of $\F$ is
$L$. The existence of such a surface was first proven by Seifert
using an algorithm on a diagram of $L$, named after him as
\emph{Seifert's algorithm}~\cite{Seifert:def}. The \emph{genus} of a
link $L$ can be defined by the minimal genus among all Seifert
surfaces of $L$, denoted by $g(L)$. A Seifert surface $\F$ of $L$
with the minimal genus $g(L)$ is called a \emph{minimal genus
Seifert surface} of $L$. A Seifert surface of $L$ is
\emph{canonical} if it is obtained from a diagram of $L$ by applying
Seifert's algorithm. Then the minimal genus among all canonical
Seifert surfaces of $L$ is called the \emph{canonical genus} of $L$,
denoted by $g_c(L)$.  A Seifert surface $\F$ of $L$ is said to be
\emph{free} if the fundamental group of the complement of $\F$,
namely, $\pi_1(\mathbb{S}^3 - \F)$ is a free group. Then the minimal
genus among all free Seifert surfaces of $L$ is called the
\emph{free genus} for $L$, denoted by $g_f (L)$. Since any canonical
Seifert surface is free, we have the following inequalities,
$$g(L)\le g_f(L) \le g_c(L).$$
There are many interesting results about the above
inequalities~\cite{Brittenham:free, Crowell:genus, KK,
Moriah:freegenus, Nakamura, sakuma:minimal}. Gabai has geometrically
shown that the minimal genus Seifert surface of $n$-pretzel links
can be found as a Murasugi sum using Thurston norms and proved that
the Seifert surfaces obtained by applying Seifert's algorithm to
the standard diagram of $L(2k_1+1,2k_2+1,\ldots ,2k_n+1)$ and
$L(2k_1,2k_2,\ldots ,2k_n)$ are minimal genus Seifert
surfaces~\cite{Gabai:genera}. There is a classical inequality
regarding the Alexander polynomial and the genus $g(L)$ of a link
$L$: G. Torres showed the following inequality,
\begin{align}
2g(L) \ge \mathrm{degree} \Delta_L - \mu +1 \label{genineqal}
\end{align}
where $\Delta_L$ is the Alexander polynomial of $L$ and $\mu$ is the
number of components of $L$ \cite{Torres:Alexander}. R. Crowell
showed that the equality in inequality~(\ref{genineqal}) holds for
alternating links~\cite{Crowell:genus}. Cimasoni has found a similar
inequality from multi-variable Alexander
polynomials~\cite{cimasoni:conway}. In fact, we can find the genera
of oriented $n$-pretzel links from the inequality~(\ref{genineqal})
and the Conway polynomial found in section~\ref{poly}, $i. e.$, we
will show that the equality in inequality~(\ref{genineqal}) holds
for all $n$-pretzel links with at least one even crossing. For
pretzel links $L(2k_1,2k_2,\ldots ,2k_n)$ with all possible
orientations, Nakagawa showed that a genus and a canonical genus are
the same~\cite{Nakagawa:Alexander}. The idea of Nakagawa
~\cite{Nakagawa:Alexander} can be extended to arbitrary $n$-pretzel
links, $i. e.$, we can show that these three genera $g(L), g_f(L)$
and $ g_c(L)$ are the same.

Some of Seifert surfaces of links feature extra structures. Seifert
surfaces obtained by plumbings annuli have been studied extensively
for the fibreness of links and surfaces \cite{Gabai:murasugi1,
Gabai:murasugi2, Gabai:genera, harer:const, Morton:alexander,
Nakamura, Rudolph:quasipositive2, Stallings:const}. Rudolph has
introduced several plumbed Seifert surfaces~\cite{Rudolph:plumbing}.
Let $A_n \subset \mathbb{S}^3$ denote an $n$-twisted unknotted
annulus. A Seifert surface $\F$ is a \emph{basket surface} if $\F =
D_2$ or if $\F = \F_0
*_{\alpha} A_n$ which can be constructed by plumbing $A_n$ to a
basket surface $\F_0$ along a proper arc $\alpha \subset D_2 \subset \F_0$
\cite{Rudolph:plumbing}. A \emph{basket number} of a link $L$,
denoted by $bk(L)$, is the minimal number of annuli used to obtain a
basket surface $\F$ such that $\partial \F=L$~\cite{BKP,
HW:plumbing}. As a consequence of the results in
section~\ref{genera} and a result~\cite[Corollary 3.3]{BKP}, we find
the basket number of $n$-pretzel links.

The outline of this paper is as follows. In section~\ref{pqr}, we
mainly focus on the classical pretzel links $L(p,q,r)$. We find
Conway polynomial of $n$-pretzel links in
section~\ref{poly}. In section~\ref{genera}, we study the genera of
$n$-pretzel links. In section~\ref{appl}, we compute the basket
number of $n$-pretzel links.

\section{Classical pretzel links $L(p,q,r)$} \label{pqr}

\subsection{Almost alternating}

One can see that $L(p,q,r)$ is alternating if $p,q,r$ have the same
signs. Since every alternating link (including any unlink) has an
almost alternating diagram, we are going to show that every
nontrivial pretzel link has an almost alternating diagram. Since the
notation depends on the choice of $+,-$ crossings, it is sufficient
to show that $L(-p,q,r)$ has an almost alternating diagram where $p,
q, r$ are positive. In particular, one might expect that $L(-1, q, r)$ is almost
alternating, but surprisingly it is also alternating.

\begin{thm} For positive integers $p ,q$ and $r$, $L(-1,q,r)$ is an alternating
link and $L(-p,q,r)$ has an almost alternating diagram. \label{alterthm}
\end{thm}
\begin{proof} One can see that $L(q,-1,r)$ is isotopic to
$L(q-2,1,r-2)$ as shown in Figure~\ref{fig2-1}. For the second part,
see Figure~\ref{fig2-2}.
\end{proof}

\begin{figure}
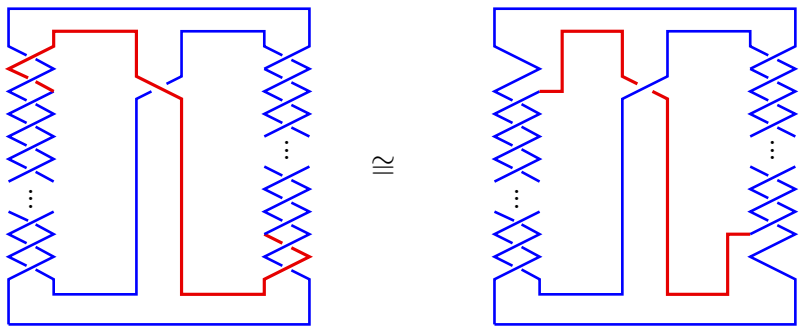

$$
\pspicture[.45](-2.7,-2.4)(2.2,2.4)
\psline(-2,-2.2)(-2,-1.6)(-1.4,-1.3)(-1.64,-1.17)
\psline(-1.76,-1.12)(-2,-1)(-1.4,-.7) \psline(-1.74,-.82)(-2,-.7)
\psline(-1.76,-1.42)(-2,-1.3)(-1.4,-1)(-1.64,-.87)
\psline(-2,-.3)(-1.4,0)(-1.64,.13) \psline(-1.4,-.3)(-1.64,-.17)
\psline(-1.76,-.12)(-2,0)(-1.4,.3)(-1.64,.43)
\psline(-1.76,.18)(-2,.3)(-1.4,.6)(-1.64,.73)
\psline(-1.76,.48)(-2,.6)(-1.4,.9)
\psline(-1.76,.78)(-2,.9)(-1.4,1.2)(-1.64,1.33)
\psline(-1.76,1.38)(-2,1.5)(-2,2)(2,2)(2,1.5)(1.4,1.2)
\psline(-1.64,-1.47)(-1.4,-1.6)(-1.4,-1.8)(-.3,-1.8)(-.3,.8)(-.1,.9)
\psline(.1,1)(.3,1.1)(.3,1.7)(1.4,1.7)(1.4,1.5)(1.64,1.38)
\psline[linewidth=1.2pt,linecolor=darkred](-1.4,.9)(-1.64,1.03)
\psline[linewidth=1.2pt,linecolor=darkred](1.4,-1)(1.64,-1.12)
\psline[linewidth=1.2pt,linecolor=darkred](-1.74,1.08)(-2,1.2)(-1.4,1.5)
(-1.4,1.7)(-.3,1.7)(-.3,1.1)(.3,.8)(.3,-1.8)(1.4,-1.8)(1.4,-1.6)(2,-1.3)(1.76,-1.18)
\psline(-2,-2.2)(2,-2.2)(2,-1.6)(1.76,-1.48)
\psline(1.64,-1.42)(1.4,-1.3)(2,-1)(1.76,-.88)
\psline(1.4,-1)(2,-.7)(1.76,-.58)
\psline(1.64,-.82)(1.4,-.7)(2,-.4)(1.76,-.28)
\psline(1.64,-.52)(1.4,-.4)(2,-.1) \psline(1.64,-.22)(1.4,-.1)
\psline(1.76,1.32)(2,1.2)(1.4,.9)(1.64,.78)
\psline(1.76,1.02)(2,.9)(1.4,.6)(1.64,.48)
\psline(1.76,.72)(2,.6)(1.4,.3) \psline(1.76,.42)(2,.3)
\psline(1.4,1.2)(1.64,1.08)
\rput[b](-1.7,-.45){$.$}\rput[b](-1.7,-.55){$.$}\rput[b](-1.7,-.65){$.$}
\rput[b](1.7,0){$.$}\rput[b](1.7,.1){$.$}\rput[b](1.7,.2){$.$}
\endpspicture \hskip .5cm \cong \hskip .5cm \pspicture[.45](-2.7,-2.4)(2.2,2.4)
\psline(-2,-2.2)(-2,-1.6)(-1.4,-1.3)(-1.64,-1.17)
\psline(-1.76,-1.12)(-2,-1)(-1.4,-.7) \psline(-1.74,-.82)(-2,-.7)
\psline(-1.76,-1.42)(-2,-1.3)(-1.4,-1)(-1.64,-.87)
\psline(-2,-.3)(-1.4,0)(-1.64,.13) \psline(-1.4,-.3)(-1.64,-.17)
\psline(-1.76,-.12)(-2,0)(-1.4,.3)(-1.64,.43)
\psline(-1.76,.18)(-2,.3)(-1.4,.6)(-1.64,.73)
\psline(-1.76,.48)(-2,.6)(-1.4,.9)
\psline(-1.76,.78)(-2,.9)(-1.4,1.2)(-2,1.5)(-2,2)(2,2)(2,1.5)(1.4,1.2)
\psline(-1.64,-1.47)(-1.4,-1.6)(-1.4,-1.8)(-.3,-1.8)(-.3,.8)(.3,1.1)(.3,1.7)(1.4,1.7)(1.4,1.5)(1.64,1.38)
\psline[linewidth=1.2pt,linecolor=darkred](-1.4,.9)(-1.1,.9)(-1.1,1.7)(-.3,1.7)(-.3,1.1)(-.1,1)
\psline[linewidth=1.2pt,linecolor=darkred](.1,.9)(.3,.8)(.3,-1.8)(1.1,-1.8)(1.1,-1)(1.4,-1)
\psline(-2,-2.2)(2,-2.2)(2,-1.6)(1.4,-1.3)(2,-1)(1.76,-.88)
\psline(1.4,-1)(2,-.7)(1.76,-.58)
\psline(1.64,-.82)(1.4,-.7)(2,-.4)(1.76,-.28)
\psline(1.64,-.52)(1.4,-.4)(2,-.1) \psline(1.64,-.22)(1.4,-.1)
\psline(1.76,1.32)(2,1.2)(1.4,.9)(1.64,.78)
\psline(1.76,1.02)(2,.9)(1.4,.6)(1.64,.48)
\psline(1.76,.72)(2,.6)(1.4,.3) \psline(1.76,.42)(2,.3)
\psline(1.4,1.2)(1.64,1.08)\rput[b](-1.7,-.45){$.$}\rput[b](-1.7,-.55){$.$}\rput[b](-1.7,-.65){$.$}
\rput[b](1.7,0){$.$}\rput[b](1.7,.1){$.$}\rput[b](1.7,.2){$.$}
\endpspicture $$
\caption{An alternating diagram of $L(q,-1,r)$.} \label{fig2-1}
\end{figure}

\begin{figure}
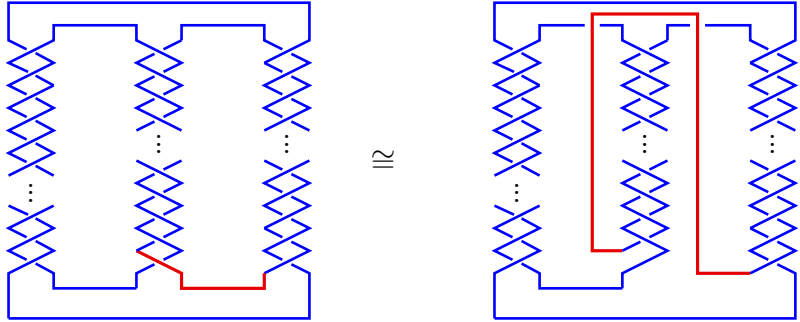
 $$
\pspicture[.45](-2.7,-2.4)(2.2,2.4)
\psline(-2,-2.2)(-2,-1.6)(-1.4,-1.3)(-1.64,-1.17)
\psline(-1.76,-1.12)(-2,-1)(-1.4,-.7) \psline(-1.74,-.82)(-2,-.7)
\psline(-1.76,-1.42)(-2,-1.3)(-1.4,-1)(-1.64,-.87)
\psline(-2,-.3)(-1.4,0)(-1.64,.13) \psline(-1.4,-.3)(-1.64,-.17)
\psline(-1.76,-.12)(-2,0)(-1.4,.3)(-1.64,.43)
\psline(-1.76,.18)(-2,.3)(-1.4,.6)(-1.64,.73)
\psline(-1.76,.48)(-2,.6)(-1.4,.9)
\psline(-1.76,.78)(-2,.9)(-1.4,1.2)(-1.64,1.33)
\psline(-1.76,1.38)(-2,1.5)(-2,2)(2,2)(2,1.5)(1.4,1.2)
\psline(-1.64,-1.47)(-1.4,-1.6)(-1.4,-1.8)(-.3,-1.8)
\psline(.3,1.5)(.3,1.7)(1.4,1.7)(1.4,1.5)(1.64,1.38)
\psline(-1.4,.9)(-1.64,1.03) \psline(1.4,-1)(1.64,-1.12)
\psline(-1.74,1.08)(-2,1.2)(-1.4,1.5)(-1.4,1.7)(-.3,1.7)(-.3,1.5)(.3,1.2)(.06,1.08)
\psline(-.06,1.32)(-.3,1.2)(.3,.9)(.06,.78)
\psline(.3,1.5)(.06,1.38) \psline(-.06,1.02)(-.3,.9)(.3,.6)(.06,.48)
\psline(-.06,.72)(-.3,.6)(.3,.3) \psline(-.06,.42)(-.3,.3)
\psline(1.4,-1.6)(2,-1.3)(1.76,-1.18)
\psline(-2,-2.2)(2,-2.2)(2,-1.6)(1.76,-1.48)
\psline(1.64,-1.42)(1.4,-1.3)(2,-1)(1.76,-.88)
\psline(1.4,-1)(2,-.7)(1.76,-.58)
\psline(1.64,-.82)(1.4,-.7)(2,-.4)(1.76,-.28)
\psline(1.64,-.52)(1.4,-.4)(2,-.1) \psline(1.64,-.22)(1.4,-.1)
\psline(1.76,1.32)(2,1.2)(1.4,.9)(1.64,.78)
\psline(1.76,1.02)(2,.9)(1.4,.6)(1.64,.48)
\psline(1.76,.72)(2,.6)(1.4,.3) \psline(1.76,.42)(2,.3)
\psline(1.4,1.2)(1.64,1.08)
\psline[linewidth=1.2pt,linecolor=darkred](1.4,-1.6)(1.4,-1.8)(.3,-1.8)(.3,-1.6)(-.3,-1.3)
\psline(-.3,-1.3)(-.06,-1.18)
\psline(-.3,-1.8)(-.3,-1.6)(-.06,-1.48)
\psline(.06,-1.42)(.3,-1.3)(-.3,-1)(-.06,-.88)
\psline(.06,-1.12)(.3,-1)(-.3,-.7)(-.06,-.58)
\psline(.06,-.82)(.3,-.7)(-.3,-.4)(-.06,-.28)
\psline(.06,-.52)(.3,-.4)(-.3,-.1) \psline(.06,-.22)(.3,-.1)
\rput[b](-1.7,-.45){$.$}\rput[b](-1.7,-.55){$.$}\rput[b](-1.7,-.65){$.$}
\rput[b](1.7,0){$.$}\rput[b](1.7,.1){$.$}\rput[b](1.7,.2){$.$}
\rput[b](0,0){$.$}\rput[b](0,.1){$.$}\rput[b](0,.2){$.$}
\endpspicture \hskip .5cm \cong \hskip .5cm \pspicture[.45](-2.7,-2.4)(2.2,2.4)
\psline(-2,-2.2)(-2,-1.6)(-1.4,-1.3)(-1.64,-1.17)
\psline(-1.76,-1.12)(-2,-1)(-1.4,-.7) \psline(-1.74,-.82)(-2,-.7)
\psline(-1.76,-1.42)(-2,-1.3)(-1.4,-1)(-1.64,-.87)
\psline(-2,-.3)(-1.4,0)(-1.64,.13) \psline(-1.4,-.3)(-1.64,-.17)
\psline(-1.76,-.12)(-2,0)(-1.4,.3)(-1.64,.43)
\psline(-1.76,.18)(-2,.3)(-1.4,.6)(-1.64,.73)
\psline(-1.76,.48)(-2,.6)(-1.4,.9)
\psline(-1.76,.78)(-2,.9)(-1.4,1.2)(-1.64,1.33)
\psline(-1.76,1.38)(-2,1.5)(-2,2)(2,2)(2,1.5)(1.4,1.2)
\psline(-1.64,-1.47)(-1.4,-1.6)(-1.4,-1.8)(-.3,-1.8)
\psline(.3,1.5)(.3,1.7)(.6,1.7)\psline(.8,1.7)(1.4,1.7)(1.4,1.5)(1.64,1.38)
\psline(-1.4,.9)(-1.64,1.03) \psline(1.4,-1)(1.64,-1.12)
\psline(-1.74,1.08)(-2,1.2)(-1.4,1.5)(-1.4,1.7)(-.8,1.7)
\psline(-.6,1.7)(-.3,1.7)(-.3,1.5)(.3,1.2)(.06,1.08)
\psline(-.06,1.32)(-.3,1.2)(.3,.9)(.06,.78)
\psline(.3,1.5)(.06,1.38) \psline(-.06,1.02)(-.3,.9)(.3,.6)(.06,.48)
\psline(-.06,.72)(-.3,.6)(.3,.3) \psline(-.06,.42)(-.3,.3)
\psline[linewidth=1.2pt,linecolor=darkred](1.4,-1.6)(.7,-1.6)(.7,1.85)(-.7,1.85)(-.7,-1.3)(-.3,-1.3)
\psline(1.4,-1.6)(2,-1.3)(1.76,-1.18)
\psline(-2,-2.2)(2,-2.2)(2,-1.6)(1.76,-1.48)
\psline(1.64,-1.42)(1.4,-1.3)(2,-1)(1.76,-.88)
\psline(1.4,-1)(2,-.7)(1.76,-.58)
\psline(1.64,-.82)(1.4,-.7)(2,-.4)(1.76,-.28)
\psline(1.64,-.52)(1.4,-.4)(2,-.1) \psline(1.64,-.22)(1.4,-.1)
\psline(1.76,1.32)(2,1.2)(1.4,.9)(1.64,.78)
\psline(1.76,1.02)(2,.9)(1.4,.6)(1.64,.48)
\psline(1.76,.72)(2,.6)(1.4,.3) \psline(1.76,.42)(2,.3)
\psline(1.4,1.2)(1.64,1.08) \psline(-.3,-1.3)(-.06,-1.18)
\psline(-.3,-1.8)(-.3,-1.6)(.3,-1.3)(-.3,-1)(-.06,-.88)
\psline(.06,-1.12)(.3,-1)(-.3,-.7)(-.06,-.58)
\psline(.06,-.82)(.3,-.7)(-.3,-.4)(-.06,-.28)
\psline(.06,-.52)(.3,-.4)(-.3,-.1) \psline(.06,-.22)(.3,-.1)
\rput[b](-1.7,-.45){$.$}\rput[b](-1.7,-.55){$.$}\rput[b](-1.7,-.65){$.$}
\rput[b](1.7,0){$.$}\rput[b](1.7,.1){$.$}\rput[b](1.7,.2){$.$}
\rput[b](0,0){$.$}\rput[b](0,.1){$.$}\rput[b](0,.2){$.$}
\endpspicture  $$
\caption{An almost alternating diagram of $L(p,-q,r)$.}
\label{fig2-2}
\end{figure}

\begin{thm}
All nontrivial pretzel knots $K(p, q, r)$ are either torus knots or
hyperbolic knots.
\end{thm}

\begin{proof} The key ingredient of theorem is that prime
alternating (almost alternating) knots are either hyperbolic or
torus knots~\cite[Corollary 2]{Menasco:alternating} (\cite[Corollary
2.4]{Adams:almost}, respectively). Since every pretzel knot has an
almost alternating diagram by Theorem~\ref{alterthm}, we need to
show that all nontrivial classical pretzel knots are prime. Since no
two of $p, q, r$ are even, there are two cases : $i)$ all of them
are odd, $ii)$ exactly one is even.

\item{$i)$} $p\equiv q\equiv r\equiv 1~(mod~2)$.
For this case, we can use the genus of $K = K(p, q, r)$. Suppose $K
= K_1 \# K_2$. Since a Seifert surface of $K$ is the punctured
torus, it has genus $1$ as described in the left top of
Figure~\ref{fig2-1}. But $1 = g(K) = g(K_1) + g(K_2)$. Thus one of
$g(K_1)$ or $g(K_2)$ has to be $0$, $i. e.,$ one of $K_i$ is
trivial. Therefore $K$ cannot be decomposed as a connected sum of
two nontrivial knots.

\item{$ii)$} Suppose that $p$ is even $i.e., p=2l$,
and $q, r$ are odd. Then it is easy to see that the left two
twisting parts form a prime tangle (except when $|p|= 2l$ and $ |q|
= 1$). The right part is an untangle, but since $r$ is odd, we can
use a result of Lickorish~\cite[Theorem 3]{Lickorish:prime} to
conclude that $K(2l,q,r)$ is prime. For the above  exceptional
cases, we can assume that $|r| =1$ because we can choose $|q| \ge
|r|$. So all possible cases are $K(2l, \pm 1, \mp 1)$, $ K(2l, 1,
1)$ and $K(2l, -1, -1)$. But the first one is the unknot and the
other two can be deformed to $K(p,q,r)$ of all odd crossings, $i.
e$, $K(2l, -1, -1) = K(2l-1,1,1)$ and $K(2l,1,1)=K(2l+1,-1,-1)$.
This completes the proof.
\end{proof}

\subsection{Prime torus pretzel knots}

The primary goal of this section is to decide which classical pretzel knots
are torus knots. For convenience, the $(m,n)$ torus link is denoted
by $T_{(m,n)}$. One can see that all $2$-string torus links are
alternating. C. Adams has conjectured that only $(3,4)$ and $(3,5)$
torus knots are almost alternating~\cite{Adams:almost}. One can see
that $K(-2, 3, 3)$ is the $(3,4)$ torus knot and $K(-2,3,5)$ is the
$(3,5)$ torus knot. Since the branched double cover of a torus link
is a Seifert fibred space with the base surface $\mathbb{S}^2$ and
at most three exceptional fibers, and the branched double cover of a
nontrivial $n$-pretzel link is a Seifert fibred space with $n$
exceptional fibers, there will be no torus knot of the form
$K(p_1,p_2,\cdots,p_n)$ for $n\ge 4$ and $|p_i|\ge 2$.

To find all torus knots, we use the Jones polynomials of $K (2l$,
$q$, $r)$ because the genera of pretzel knots tell us that no
$K(p,q,r)$, with $p, q, r$ all odd, is a torus knot except the
unknot and trefoil, and it is known that $K(p$,$-1$,$1)$ is the
unknot and $K(\pm 1, \pm 1, \pm 1)$ are trefoils, which are the only
torus knots of genus 1. Remark that the genus of an $(m,n)$ torus
knot is $(m-1)(n-1)/2$. The Jones polynomial of an $(m,n)$ torus
link $(m\le n)$ is given by equation~(\ref{modd}) if $m$ is odd, by
equation~(\ref{meven}) if $4 \le m$ is even, and by
equation~(\ref{m2}) if $m=2$ and $n$ is even. This is due to the
original work by Jones~\cite{Jones:hecke} but still there is no
combinatorial proof for these formulae.

\begin{align}
- t^{(m-1)(n-1)/2}&[t^{m+n-2} +t^{m+n-4}+ \cdots +t^{n+1}-t^{m-1}-
\cdots -t^2 -1], \label{modd}
\\
- t^{(m-1)(n-1)/2}&[t^{m+n-2}+ t^{m+n-4}+ \cdots +t^{n}-t^{n-1}-
\cdots -t^2 -1], \label{meven}
\\
- t^{(n-1)/2}&[t^n - t^{n-1} + t^{n-2} - \cdots -t^3 +t^2 +1].
\label{m2}
\end{align}

Using a formula for the Jones polynomials of $n$-pretzel knots in
\cite{Landvoy:Jones}, we find the following lemma. Since the Jones
polynomial of the mirror image $\overline{L}$ of $L$ can be found by
$V_{\overline L}(t)=V_L(t^{-1})$, we may assume $q, r$ are positive
integers.

\begin{lem} Let $l, q, r$ be positive integers. Let $k = 2l+q+r$.

\begin{align*}
V_{K(2, 1, r)} &= t^{(r+1)/2-(2+1)} (t^{r+2+1} - 2t^{r+2} + 2t^{r+1}
- \cdots + 2t^3 -t^2 +t - 1),  \\
V_{K(2l, q, r)} &= t^{(q+r)/2-(2l+1)} (t^k -2t^{k-1} + 3t^{k-2}
-4t^{k-3} +
\cdots -3t^2 +t -1), \hskip .2cm \mathrm{if}~l \ge 1, \\
V_{K(2l, -q, r)} &= - t^{(-4l-3q+r)/2} (t^{q+r} - t^{q+r-1} + \cdots
- t + 1) \hskip 1cm  \mathrm{if}~  q > 1, \\ V_{K(-2, 1, r)} &= -
t^{(r+1)/2} (t^{r+2} - t^{r+1} +t^r - \cdots + t^3 -t^2 - 1),
\\
V_{K(-2, 3, 3)} &= - t^3 (t^5 - t^2 - 1),
\\
V_{K(-2, 3, 5)} &= - t^4 (t^6 - t^2 - 1),
\\
V_{K(-2, 3, r)} &= - t^{(3+r)/2} (t^{3+r-2} - t^r + \cdots -t^2 - 1)
\hskip 2cm  \mathrm{if}~ r \ge 7,
\\
V_{K(-2, q, r)} &= - t^{(q+r)/2} (- t^{q+r-1} +2t^{q+r-2} -\cdots
-t^2 - 1) \hskip .7cm  \mathrm{if}~  q, r \ge 5,
\\
V_{K(-2l, q, r)} &= - t^{(q+r)/2} ( at^{*} + \cdots \pm t \mp 1)
\hskip 3.6cm \mathrm{if}~  l, q, r > 1.
\end{align*} \label{jonespolylem}
\end{lem}

By comparing Jones polynomials of pretzel knots in
Lemma~\ref{jonespolylem} and Jones polynomials of torus knots in
equation~(\ref{modd}), (\ref{meven}) and (\ref{m2}), we find the
following theorem.

\begin{thm} The following are the only nontrivial pretzel
knots which are torus knots.
\item{$\mathrm{1)}$} $K(p, \pm 1, \mp 1)$ are unknots for all $p$.
\item{$\mathrm{2)}$} $K(\pm 1, \pm 1, \pm 1)$ are $(2, \pm 3)$ torus knots.
\item{$\mathrm{3)}$} $K(\pm 2, \mp 1, \pm r)$ are $(2, \pm r \mp 2)$ torus knots.
\item{$\mathrm{4)}$} $K(\mp 2, \pm 3, \pm 3)$, $K(\mp 2, \pm 3, \pm 5)$ are
$(3,\pm 4)$, $(3,\pm 5)$ torus knots, respectively.
\label{toruspretzel} \end{thm}
\begin{proof}
We only need to consider $K(2l, q, r)$. We can see
that $K(2$,$ -1$,$ r)$ can be deformed to $K(0,$$ 1,$$ r - 2)$ by a
move shown in Figure~\ref{fig2-1}. The coefficient of $t^1$ and the
second leading coefficient of the Jones polynomial of a torus knot
are zero, but by Lemma~\ref{jonespolylem} these are possible only
for $K(-2, 3, 3)$, $K(-2,3,5)$ and their mirror images. But the
number of terms in the Jones polynomials of these knots is $3$, and
only $(3,n)$ torus knots have this property. By comparing the terms
of the highest degree, we conclude that $K(\mp 2, \pm 3, \pm3)$ and
$K(\mp 2, \pm 3, \pm5)$ are the remaining non-alternating torus
knots.
\end{proof}

\subsection{Minimal genus Seifert surfaces}

When one applies Seifert's algorithm to a diagram of a link $L$, in
general one may not get a minimal genus Seifert surface. In fact,
Moriah found infinitely many knots which have no diagram on which
Seifert's algorithm produces a minimal genus
surface~\cite{Moriah:freegenus}. But it is known that a minimal
genus Seifert surface can be obtained from an alternating diagram by
applying Seifert's algorithm~\cite{Murasugi:genus} and more generally, alternative
links~\cite{Kauffman:comb}. We prove that the Seifert surface
obtained by applying Seifert's algorithm to the diagram in
Figure~\ref{fig4} of a pretzel knot $K(p,q,r)$ is a minimal genus
Seifert surface. Since $K(2l,q,r)$ and its mirror image are
alternating, without loss of a generality, we only need to find
Alexander polynomials of $K(-2l,q,r)$ and $K(-2l,q,-r)$.

\begin{lem} Let $l, q, r$ be positive integers.
\begin{align*}
\Delta _{K(-2l, q, r)}(t) &= t^{-(q+r)/2} ( l t^{q+r} -
(2l-1)t^{q+r-1}
+ \cdots - (2l-1)t + l ), \\
\Delta _{K(-2l, q, - r)}(t)  &= t^{-(q+r-2)/2} (  t^{q+r-2} -
2t^{q+r-3} + \cdots - 2t + 1 ).
\end{align*} \label{alexanderlem}
\end{lem}
\begin{proof}
One can prove inductively the lemma by the following recurrence
formulae which come from the skein relations, and the formulae for
the Alexander polynomial of the $(2,p)$ torus links.
\begin{align}\notag
\Delta_{T_{(2,\pm p)}} (t) &= t^{(1-p)/2} ( t^{p-1} - t^{p-2} +
\cdots - t + 1) {\rm \pol if \pol} p {\pol \rm \pol is \pol
odd},\\\notag \Delta_{T_{(2,\pm p)}} (t) &= t^{(1-p)/2} ( -t^{p-1} +
t^{p-2} + \cdots - t + 1) {\rm \pol if \pol} p {\pol \rm \pol is
\pol even},\\\notag \Delta _{K(-2, q, \pm r)} (t)  &= \Delta
_{T_{(2,q)}}(t) \Delta_{T_{(2,r)}}(t) +(t^{-1/2}- t^{1/2}) \Delta
_{T_{(2,q \pm r)}}(t) , \\\notag \Delta_{K(-2l, q, \pm r)} (t) &=
\Delta _{K(-2(l-1), q, \pm r)}(t) +(t^{-1/2}- t^{1/2}) \Delta
_{T_{(2,q \pm r)}}(t) .
\end{align}
\end{proof}

\begin{figure}
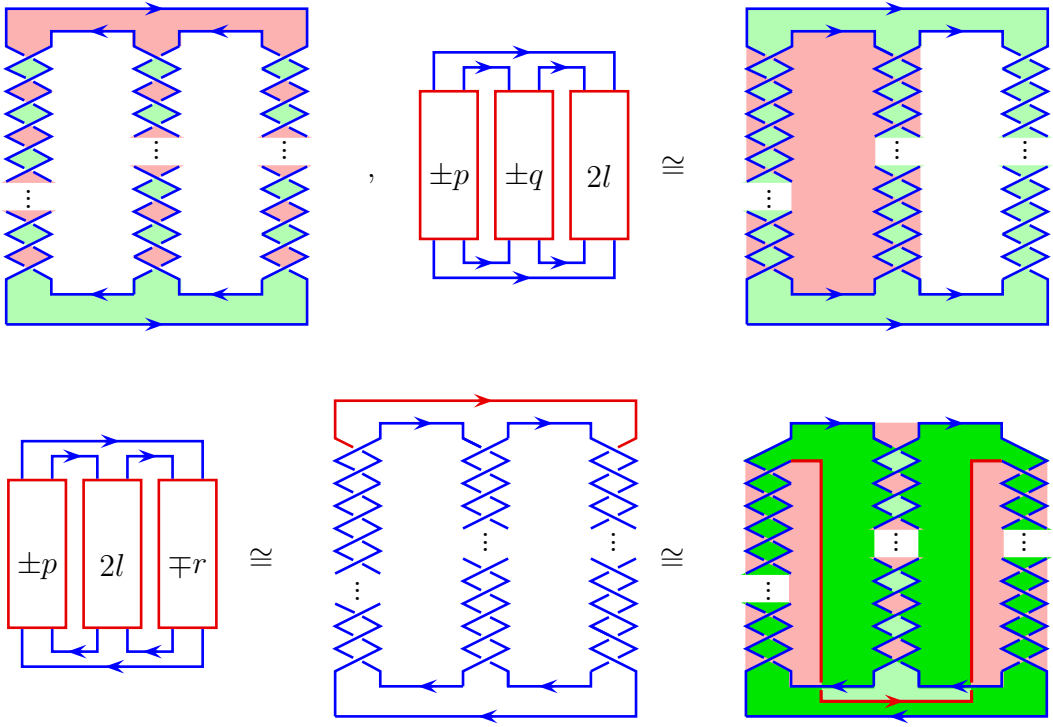
$$ \pspicture[.45](-2.3,-2.4)(2.6,2.4)
\pspolygon[linecolor=lightgreen,fillstyle=solid,fillcolor=lightgreen](-1.99,-2.19)(-1.99,-1.61)(-1.7,-1.46)(-1.41,-1.61)(-1.41,-1.81)
(-.31,-1.81)(-.31,-1.61)(0,-1.46)(.31,-1.61)(.31,-1.81)(1.41,-1.81)(1.41,-1.61)(1.7,-1.46)(1.99,-1.61)(1.99,-2.19)
\pspolygon[linecolor=lightred,fillstyle=solid,fillcolor=lightred](-1.99,1.99)(-1.99,1.51)(-1.7,1.36)(-1.41,1.51)(-1.41,1.71)
(-.31,1.71)(-.31,1.51)(0,1.36)(.31,1.51)(.31,1.71)(1.41,1.71)(1.41,1.51)(1.7,1.36)(1.99,1.51)(1.99,1.99)
\pspolygon[linecolor=lightred,fillstyle=solid,fillcolor=lightred](-1.99,-1.3)(-1.7,-1.16)(-1.41,-1.3)(-1.7,-1.44)
\pspolygon[linecolor=lightred,fillstyle=solid,fillcolor=lightred](-.29,-1.3)(0,-1.16)(.29,-1.3)(0,-1.44)
\pspolygon[linecolor=lightred,fillstyle=solid,fillcolor=lightred](1.99,-1.3)(1.7,-1.16)(1.41,-1.3)(1.7,-1.44)
\pspolygon[linecolor=lightgreen,fillstyle=solid,fillcolor=lightgreen](-1.99,-1)(-1.7,-.86)(-1.41,-1)(-1.7,-1.14)
\pspolygon[linecolor=lightgreen,fillstyle=solid,fillcolor=lightgreen](-.29,-1)(0,-.86)(.29,-1)(0,-1.14)
\pspolygon[linecolor=lightgreen,fillstyle=solid,fillcolor=lightgreen](1.99,-1)(1.7,-.86)(1.41,-1)(1.7,-1.14)
\pspolygon[linecolor=lightred,fillstyle=solid,fillcolor=lightred](-1.99,-.7)(-1.41,-.7)(-1.7,-.84)
\pspolygon[linecolor=lightred,fillstyle=solid,fillcolor=lightred](-.29,-.7)(0,-.56)(.29,-.7)(0,-.84)
\pspolygon[linecolor=lightred,fillstyle=solid,fillcolor=lightred](1.99,-.7)(1.7,-.56)(1.41,-.7)(1.7,-.84)
\pspolygon[linecolor=lightgreen,fillstyle=solid,fillcolor=lightgreen](-.29,-.4)(0,-.26)(.29,-.4)(0,-.54)
\pspolygon[linecolor=lightgreen,fillstyle=solid,fillcolor=lightgreen](1.99,-.4)(1.7,-.26)(1.41,-.4)(1.7,-.54)
\pspolygon[linecolor=lightred,fillstyle=solid,fillcolor=lightred](-.29,-.1)(.29,-.1)(0,-.24)
\pspolygon[linecolor=lightred,fillstyle=solid,fillcolor=lightred](1.99,-.1)(1.41,-.1)(1.7,-.24)
\pspolygon[linecolor=lightgreen,fillstyle=solid,fillcolor=lightgreen](-1.99,1.2)(-1.7,1.06)(-1.41,1.2)(-1.7,1.34)
\pspolygon[linecolor=lightgreen,fillstyle=solid,fillcolor=lightgreen](-.29,1.2)(0,1.06)(.29,1.2)(0,1.34)
\pspolygon[linecolor=lightgreen,fillstyle=solid,fillcolor=lightgreen](1.99,1.2)(1.7,1.06)(1.41,1.2)(1.7,1.34)
\pspolygon[linecolor=lightred,fillstyle=solid,fillcolor=lightred](-1.99,.9)(-1.7,.76)(-1.41,.9)(-1.7,1.04)
\pspolygon[linecolor=lightred,fillstyle=solid,fillcolor=lightred](-.29,.9)(0,.76)(.29,.9)(0,1.04)
\pspolygon[linecolor=lightred,fillstyle=solid,fillcolor=lightred](1.99,.9)(1.7,.76)(1.41,.9)(1.7,1.04)
\pspolygon[linecolor=lightgreen,fillstyle=solid,fillcolor=lightgreen](-1.99,.6)(-1.7,.46)(-1.41,.6)(-1.7,.74)
\pspolygon[linecolor=lightgreen,fillstyle=solid,fillcolor=lightgreen](-.29,.6)(0,.46)(.29,.6)(0,.74)
\pspolygon[linecolor=lightgreen,fillstyle=solid,fillcolor=lightgreen](1.99,.6)(1.7,.46)(1.41,.6)(1.7,.74)
\pspolygon[linecolor=lightred,fillstyle=solid,fillcolor=lightred](-1.99,.3)(-1.7,.16)(-1.41,.3)(-1.7,.44)
\pspolygon[linecolor=lightred,fillstyle=solid,fillcolor=lightred](-.29,.3)(.29,.3)(0,.44)
\pspolygon[linecolor=lightred,fillstyle=solid,fillcolor=lightred](1.99,.3)(1.41,.3)(1.7,.44)
\pspolygon[linecolor=lightgreen,fillstyle=solid,fillcolor=lightgreen](-1.99,0)(-1.7,-.14)(-1.41,0)(-1.7,.14)
\pspolygon[linecolor=lightred,fillstyle=solid,fillcolor=lightred](-1.99,-.3)(-1.41,-.3)(-1.7,-.16)
\psline(-2,-2.2)(-2,-1.6)(-1.4,-1.3)(-1.64,-1.17)
\psline(-1.76,-1.12)(-2,-1)(-1.4,-.7) \psline(-1.74,-.82)(-2,-.7)
\psline(-1.76,-1.42)(-2,-1.3)(-1.4,-1)(-1.64,-.87)
\psline(-2,-.3)(-1.4,0)(-1.64,.13) \psline(-1.4,-.3)(-1.64,-.17)
\psline(-1.76,-.12)(-2,0)(-1.4,.3)(-1.64,.43)
\psline(-1.76,.18)(-2,.3)(-1.4,.6)(-1.64,.73)
\psline(-1.76,.48)(-2,.6)(-1.4,.9)
\psline(-1.76,.78)(-2,.9)(-1.4,1.2)(-1.64,1.33)
\psline(-1.76,1.38)(-2,1.5)(-2,2)(2,2)(2,1.5)(1.4,1.2)
\psline(-1.64,-1.47)(-1.4,-1.6)(-1.4,-1.8)(-.3,-1.8)
\psline(.3,1.5)(.3,1.7)(1.4,1.7)(1.4,1.5)(1.64,1.38)
\psline(-1.4,.9)(-1.64,1.03) \psline(1.4,-1)(1.64,-1.12)
\psline(-1.74,1.08)(-2,1.2)(-1.4,1.5)(-1.4,1.7)(-.3,1.7)(-.3,1.5)(.3,1.2)(.06,1.08)
\psline(-.06,1.32)(-.3,1.2)(.3,.9)(.06,.78)
\psline(.3,1.5)(.06,1.38) \psline(-.06,1.02)(-.3,.9)(.3,.6)(.06,.48)
\psline(-.06,.72)(-.3,.6)(.3,.3) \psline(-.06,.42)(-.3,.3)
\psline(1.4,-1.6)(2,-1.3)(1.76,-1.18)
\psline(-2,-2.2)(2,-2.2)(2,-1.6)(1.76,-1.48)
\psline(1.64,-1.42)(1.4,-1.3)(2,-1)(1.76,-.88)
\psline(1.4,-1)(2,-.7)(1.76,-.58)
\psline(1.64,-.82)(1.4,-.7)(2,-.4)(1.76,-.28)
\psline(1.64,-.52)(1.4,-.4)(2,-.1) \psline(1.64,-.22)(1.4,-.1)
\psline(1.76,1.32)(2,1.2)(1.4,.9)(1.64,.78)
\psline(1.76,1.02)(2,.9)(1.4,.6)(1.64,.48)
\psline(1.76,.72)(2,.6)(1.4,.3) \psline(1.76,.42)(2,.3)
\psline(1.4,1.2)(1.64,1.08)
\psline(1.4,-1.6)(1.4,-1.8)(.3,-1.8)(.3,-1.6)(-.3,-1.3)\psline(-.3,-1.3)(-.06,-1.18)
\psline(-.3,-1.8)(-.3,-1.6)(-.06,-1.48)
\psline(.06,-1.42)(.3,-1.3)(-.3,-1)(-.06,-.88)
\psline(.06,-1.12)(.3,-1)(-.3,-.7)(-.06,-.58)
\psline(.06,-.82)(.3,-.7)(-.3,-.4)(-.06,-.28)
\psline(.06,-.52)(.3,-.4)(-.3,-.1) \psline(.06,-.22)(.3,-.1)
\psline[arrowscale=1.5]{->}(-.1,2)(.1,2)
\psline[arrowscale=1.5]{->}(-.1,-2.2)(.1,-2.2)
\psline[arrowscale=1.5]{->}(-.7,-1.8)(-.9,-1.8)
\psline[arrowscale=1.5]{->}(.9,-1.8)(.7,-1.8)
\psline[arrowscale=1.5]{->}(-.7,1.7)(-.9,1.7)
\psline[arrowscale=1.5]{->}(.9,1.7)(.7,1.7)
\rput[b](-1.7,-.45){$.$}\rput[b](-1.7,-.55){$.$}\rput[b](-1.7,-.65){$.$}
\rput[b](1.7,0){$.$}\rput[b](1.7,.1){$.$}\rput[b](1.7,.2){$.$}
\rput[b](0,0){$.$}\rput[b](0,.1){$.$}\rput[b](0,.2){$.$}
\endpspicture \hskip .2cm, \hskip .2cm \pspicture[.45](-1.7,-1.6)(1.7,1.6)
\psframe[linecolor=darkred](-1.4,-1)(-.6,1) \rput[t](-1,0){$\pm p$}
\psframe[linecolor=darkred](-.4,-1)(.4,1) \rput[t](0,0){$\pm q$}
\psframe[linecolor=darkred](.6,-1)(1.4,1) \rput[t](1,0){$2l$}
\psline(-1.2,1)(-1.2,1.5)(1.2,1.5)(1.2,1)
\psline(-.8,1)(-.8,1.3)(-.2,1.3)(-.2,1)
\psline(.8,1)(.8,1.3)(.2,1.3)(.2,1)
\psline(-1.2,-1)(-1.2,-1.5)(1.2,-1.5)(1.2,-1)
\psline(-.8,-1)(-.8,-1.3)(-.2,-1.3)(-.2,-1)
\psline(.8,-1)(.8,-1.3)(.2,-1.3)(.2,-1)
\psline[arrowscale=1.5]{->}(-.1,1.5)(.1,1.5)
\psline[arrowscale=1.5]{->}(-.1,-1.5)(.1,-1.5)
\psline[arrowscale=1.5]{->}(.4,1.3)(.6,1.3)
\psline[arrowscale=1.5]{->}(.4,-1.3)(.6,-1.3)
\psline[arrowscale=1.5]{->}(-.6,1.3)(-.4,1.3)
\psline[arrowscale=1.5]{->}(-.6,-1.3)(-.4,-1.3)
\endpspicture \cong \pspicture[.45](-2.7,-2.4)(2.2,2.4)
\pspolygon[linecolor=lightgreen,fillstyle=solid,fillcolor=lightgreen](-1.99,-2.19)(-1.99,-1.61)(-1.7,-1.46)(-1.41,-1.61)(-1.41,-1.81)
(-.31,-1.81)(-.31,-1.61)(0,-1.46)(.31,-1.61)(.31,-1.81)(1.41,-1.81)(1.41,-1.61)(1.7,-1.46)(1.99,-1.61)(1.99,-2.19)
\pspolygon[linecolor=lightgreen,fillstyle=solid,fillcolor=lightgreen](-1.99,1.99)(-1.99,1.51)(-1.7,1.36)(-1.41,1.51)(-1.41,1.71)
(-.31,1.71)(-.31,1.51)(0,1.36)(.31,1.51)(.31,1.71)(1.41,1.71)(1.41,1.51)(1.7,1.36)(1.99,1.51)(1.99,1.99)
\pspolygon[linecolor=lightgreen,fillstyle=solid,fillcolor=lightgreen](-1.99,-1.3)(-1.7,-1.16)(-1.41,-1.3)(-1.7,-1.44)
\pspolygon[linecolor=lightgreen,fillstyle=solid,fillcolor=lightgreen](-.29,-1.3)(0,-1.16)(.29,-1.3)(0,-1.44)
\pspolygon[linecolor=lightgreen,fillstyle=solid,fillcolor=lightgreen](1.99,-1.3)(1.7,-1.16)(1.41,-1.3)(1.7,-1.44)
\pspolygon[linecolor=lightgreen,fillstyle=solid,fillcolor=lightgreen](-1.99,-1)(-1.7,-.86)(-1.41,-1)(-1.7,-1.14)
\pspolygon[linecolor=lightgreen,fillstyle=solid,fillcolor=lightgreen](-.29,-1)(0,-.86)(.29,-1)(0,-1.14)
\pspolygon[linecolor=lightgreen,fillstyle=solid,fillcolor=lightgreen](1.99,-1)(1.7,-.86)(1.41,-1)(1.7,-1.14)
\pspolygon[linecolor=lightgreen,fillstyle=solid,fillcolor=lightgreen](-1.99,-.7)(-1.41,-.7)(-1.7,-.84)
\pspolygon[linecolor=lightgreen,fillstyle=solid,fillcolor=lightgreen](-.29,-.7)(0,-.56)(.29,-.7)(0,-.84)
\pspolygon[linecolor=lightgreen,fillstyle=solid,fillcolor=lightgreen](1.99,-.7)(1.7,-.56)(1.41,-.7)(1.7,-.84)
\pspolygon[linecolor=lightgreen,fillstyle=solid,fillcolor=lightgreen](-.29,-.4)(0,-.26)(.29,-.4)(0,-.54)
\pspolygon[linecolor=lightgreen,fillstyle=solid,fillcolor=lightgreen](1.99,-.4)(1.7,-.26)(1.41,-.4)(1.7,-.54)
\pspolygon[linecolor=lightgreen,fillstyle=solid,fillcolor=lightgreen](-.29,-.1)(.29,-.1)(0,-.24)
\pspolygon[linecolor=lightgreen,fillstyle=solid,fillcolor=lightgreen](1.99,-.1)(1.41,-.1)(1.7,-.24)
\pspolygon[linecolor=lightgreen,fillstyle=solid,fillcolor=lightgreen](-1.99,1.2)(-1.7,1.06)(-1.41,1.2)(-1.7,1.34)
\pspolygon[linecolor=lightgreen,fillstyle=solid,fillcolor=lightgreen](-.29,1.2)(0,1.06)(.29,1.2)(0,1.34)
\pspolygon[linecolor=lightgreen,fillstyle=solid,fillcolor=lightgreen](1.99,1.2)(1.7,1.06)(1.41,1.2)(1.7,1.34)
\pspolygon[linecolor=lightgreen,fillstyle=solid,fillcolor=lightgreen](-1.99,.9)(-1.7,.76)(-1.41,.9)(-1.7,1.04)
\pspolygon[linecolor=lightgreen,fillstyle=solid,fillcolor=lightgreen](-.29,.9)(0,.76)(.29,.9)(0,1.04)
\pspolygon[linecolor=lightgreen,fillstyle=solid,fillcolor=lightgreen](1.99,.9)(1.7,.76)(1.41,.9)(1.7,1.04)
\pspolygon[linecolor=lightgreen,fillstyle=solid,fillcolor=lightgreen](-1.99,.6)(-1.7,.46)(-1.41,.6)(-1.7,.74)
\pspolygon[linecolor=lightgreen,fillstyle=solid,fillcolor=lightgreen](-.29,.6)(0,.46)(.29,.6)(0,.74)
\pspolygon[linecolor=lightgreen,fillstyle=solid,fillcolor=lightgreen](1.99,.6)(1.7,.46)(1.41,.6)(1.7,.74)
\pspolygon[linecolor=lightgreen,fillstyle=solid,fillcolor=lightgreen](-1.99,.3)(-1.7,.16)(-1.41,.3)(-1.7,.44)
\pspolygon[linecolor=lightgreen,fillstyle=solid,fillcolor=lightgreen](-.29,.3)(.29,.3)(0,.44)
\pspolygon[linecolor=lightgreen,fillstyle=solid,fillcolor=lightgreen](1.99,.3)(1.41,.3)(1.7,.44)
\pspolygon[linecolor=lightgreen,fillstyle=solid,fillcolor=lightgreen](-1.99,0)(-1.7,-.14)(-1.41,0)(-1.7,.14)
\pspolygon[linecolor=lightgreen,fillstyle=solid,fillcolor=lightgreen](-1.99,-.3)(-1.41,-.3)(-1.7,-.16)
\pspolygon[linecolor=lightred,fillstyle=solid,fillcolor=lightred](-1.39,1.69)(-.31,1.69)(-.31,-1.79)(-1.39,-1.79)
\pspolygon[linecolor=lightred,fillstyle=solid,fillcolor=lightred](-1.99,-1.59)(-1.71,-1.45)(-1.99,-1.31)
\pspolygon[linecolor=lightred,fillstyle=solid,fillcolor=lightred](-1.99,-1.29)(-1.71,-1.15)(-1.99,-1.01)
\pspolygon[linecolor=lightred,fillstyle=solid,fillcolor=lightred](-1.99,-.99)(-1.71,-.85)(-1.99,-.71)
\pspolygon[linecolor=lightred,fillstyle=solid,fillcolor=lightred](-1.99,1.49)(-1.71,1.35)(-1.99,1.21)
\pspolygon[linecolor=lightred,fillstyle=solid,fillcolor=lightred](-1.99,1.19)(-1.71,1.05)(-1.99,.91)
\pspolygon[linecolor=lightred,fillstyle=solid,fillcolor=lightred](-1.99,.89)(-1.71,.75)(-1.99,.61)
\pspolygon[linecolor=lightred,fillstyle=solid,fillcolor=lightred](-1.99,.59)(-1.71,.45)(-1.99,.31)
\pspolygon[linecolor=lightred,fillstyle=solid,fillcolor=lightred](-1.99,.29)(-1.71,.15)(-1.99,.01)
\pspolygon[linecolor=lightred,fillstyle=solid,fillcolor=lightred](-1.99,-.01)(-1.71,-.15)(-1.99,-.29)
\pspolygon[linecolor=lightred,fillstyle=solid,fillcolor=lightred](-1.41,-1.59)(-1.69,-1.45)(-1.41,-1.31)
\pspolygon[linecolor=lightred,fillstyle=solid,fillcolor=lightred](-1.41,-1.29)(-1.69,-1.15)(-1.41,-1.01)
\pspolygon[linecolor=lightred,fillstyle=solid,fillcolor=lightred](-1.41,-.99)(-1.69,-.85)(-1.41,-.71)
\pspolygon[linecolor=lightred,fillstyle=solid,fillcolor=lightred](-1.41,1.49)(-1.69,1.35)(-1.41,1.21)
\pspolygon[linecolor=lightred,fillstyle=solid,fillcolor=lightred](-1.41,1.19)(-1.69,1.05)(-1.41,.91)
\pspolygon[linecolor=lightred,fillstyle=solid,fillcolor=lightred](-1.41,.89)(-1.69,.75)(-1.41,.61)
\pspolygon[linecolor=lightred,fillstyle=solid,fillcolor=lightred](-1.41,.59)(-1.69,.45)(-1.41,.31)
\pspolygon[linecolor=lightred,fillstyle=solid,fillcolor=lightred](-1.41,.29)(-1.69,.15)(-1.41,.01)
\pspolygon[linecolor=lightred,fillstyle=solid,fillcolor=lightred](-1.41,-.01)(-1.69,-.15)(-1.41,-.29)
\pspolygon[linecolor=lightred,fillstyle=solid,fillcolor=lightred](-.29,1.49)(-.01,1.35)(-.29,1.21)
\pspolygon[linecolor=lightred,fillstyle=solid,fillcolor=lightred](-.29,1.19)(-.01,1.05)(-.29,.91)
\pspolygon[linecolor=lightred,fillstyle=solid,fillcolor=lightred](-.29,.89)(-.01,.75)(-.29,.61)
\pspolygon[linecolor=lightred,fillstyle=solid,fillcolor=lightred](-.29,.59)(-.01,.45)(-.29,.31)
\pspolygon[linecolor=lightred,fillstyle=solid,fillcolor=lightred](.29,1.49)(.01,1.35)(.29,1.21)
\pspolygon[linecolor=lightred,fillstyle=solid,fillcolor=lightred](.29,1.19)(.01,1.05)(.29,.91)
\pspolygon[linecolor=lightred,fillstyle=solid,fillcolor=lightred](.29,.89)(.01,.75)(.29,.61)
\pspolygon[linecolor=lightred,fillstyle=solid,fillcolor=lightred](.29,.59)(.01,.45)(.29,.31)
\pspolygon[linecolor=lightred,fillstyle=solid,fillcolor=lightred](-.29,-1.59)(-.01,-1.45)(-.29,-1.31)
\pspolygon[linecolor=lightred,fillstyle=solid,fillcolor=lightred](-.29,-1.29)(-.01,-1.15)(-.29,-1.01)
\pspolygon[linecolor=lightred,fillstyle=solid,fillcolor=lightred](-.29,-.99)(-.01,-.85)(-.29,-.71)
\pspolygon[linecolor=lightred,fillstyle=solid,fillcolor=lightred](-.29,-.69)(-.01,-.55)(-.29,-.41)
\pspolygon[linecolor=lightred,fillstyle=solid,fillcolor=lightred](-.29,-.39)(-.01,-.25)(-.29,-.11)
\pspolygon[linecolor=lightred,fillstyle=solid,fillcolor=lightred](.29,-1.59)(.01,-1.45)(.29,-1.31)
\pspolygon[linecolor=lightred,fillstyle=solid,fillcolor=lightred](.29,-1.29)(.01,-1.15)(.29,-1.01)
\pspolygon[linecolor=lightred,fillstyle=solid,fillcolor=lightred](.29,-.99)(.01,-.85)(.29,-.71)
\pspolygon[linecolor=lightred,fillstyle=solid,fillcolor=lightred](.29,-.69)(.01,-.55)(.29,-.41)
\pspolygon[linecolor=lightred,fillstyle=solid,fillcolor=lightred](.29,-.39)(.01,-.25)(.29,-.11)
\psline(-2,-2.2)(-2,-1.6)(-1.4,-1.3)(-1.64,-1.17)
\psline(-1.76,-1.12)(-2,-1)(-1.4,-.7) \psline(-1.74,-.82)(-2,-.7)
\psline(-1.76,-1.42)(-2,-1.3)(-1.4,-1)(-1.64,-.87)
\psline(-2,-.3)(-1.4,0)(-1.64,.13) \psline(-1.4,-.3)(-1.64,-.17)
\psline(-1.76,-.12)(-2,0)(-1.4,.3)(-1.64,.43)
\psline(-1.76,.18)(-2,.3)(-1.4,.6)(-1.64,.73)
\psline(-1.76,.48)(-2,.6)(-1.4,.9)
\psline(-1.76,.78)(-2,.9)(-1.4,1.2)(-1.64,1.33)
\psline(-1.76,1.38)(-2,1.5)(-2,2)(2,2)(2,1.5)(1.4,1.2)
\psline(-1.64,-1.47)(-1.4,-1.6)(-1.4,-1.8)(-.3,-1.8)
\psline(.3,1.5)(.3,1.7)(1.4,1.7)(1.4,1.5)(1.64,1.38)
\psline(-1.4,.9)(-1.64,1.03) \psline(1.4,-1)(1.64,-1.12)
\psline(-1.74,1.08)(-2,1.2)(-1.4,1.5)(-1.4,1.7)(-.3,1.7)(-.3,1.5)(-.06,1.38)
\psline(.3,1.5)(-.3,1.2)(-.06,1.08)
\psline(.06,1.32)(.3,1.2)(-.3,.9)(-.06,.78)
\psline(.06,1.02)(.3,.9)(-.3,.6)(-.06,.48)
\psline(-.3,1.5)(-.06,1.38) \psline(.06,.72)(.3,.6)(-.3,.3)
\psline(.06,.42)(.3,.3) \psline(1.4,-1.6)(2,-1.3)(1.76,-1.18)
\psline(-2,-2.2)(2,-2.2)(2,-1.6)(1.76,-1.48)
\psline(1.64,-1.42)(1.4,-1.3)(2,-1)(1.76,-.88)
\psline(1.4,-1)(2,-.7)(1.76,-.58)
\psline(1.64,-.82)(1.4,-.7)(2,-.4)(1.76,-.28)
\psline(1.64,-.52)(1.4,-.4)(2,-.1) \psline(1.64,-.22)(1.4,-.1)
\psline(1.76,1.32)(2,1.2)(1.4,.9)(1.64,.78)
\psline(1.76,1.02)(2,.9)(1.4,.6)(1.64,.48)
\psline(1.76,.72)(2,.6)(1.4,.3) \psline(1.76,.42)(2,.3)
\psline(1.4,1.2)(1.64,1.08)
\psline(1.4,-1.6)(1.4,-1.8)(.3,-1.8)(.3,-1.6)
\psline(.3,-1.8)(.3,-1.6)(.06,-1.48)
\psline(-.3,-1.8)(-.3,-1.6)(.3,-1.3)(.06,-1.18)
\psline(-.06,-1.42)(-.3,-1.3)(.3,-1)(.06,-.88)
\psline(-.06,-1.12)(-.3,-1)(.3,-.7)(.06,-.58)
\psline(-.06,-.82)(-.3,-.7)(.3,-.4)(.06,-.28)
\psline(-.06,-.52)(-.3,-.4)(.3,-.1) \psline(-.06,-.22)(-.3,-.1)
\psline[arrowscale=1.5]{->}(-.1,2)(.1,2)
\psline[arrowscale=1.5]{->}(-.1,-2.2)(.1,-2.2)
\psline[arrowscale=1.5]{->}(-.9,-1.8)(-.7,-1.8)
\psline[arrowscale=1.5]{->}(.7,-1.8)(.9,-1.8)
\psline[arrowscale=1.5]{->}(-.9,1.7)(-.7,1.7)
\psline[arrowscale=1.5]{->}(.7,1.7)(.9,1.7)
\rput[b](-1.7,-.45){$.$}\rput[b](-1.7,-.55){$.$}\rput[b](-1.7,-.65){$.$}
\rput[b](1.7,0){$.$}\rput[b](1.7,.1){$.$}\rput[b](1.7,.2){$.$}
\rput[b](0,0){$.$}\rput[b](0,.1){$.$}\rput[b](0,.2){$.$}
\endpspicture$$

$$ \pspicture[.45](-1.7,-2)(1.7,2)
\psframe[linecolor=darkred](-1.4,-1)(-.6,1) \rput[t](-1,0){$\pm p$}
\psframe[linecolor=darkred](-.4,-1)(.4,1) \rput[t](0,0){$2l$}
\psframe[linecolor=darkred](.6,-1)(1.4,1) \rput[t](1,0){$\mp r$}
\psline(-1.2,1)(-1.2,1.5)(1.2,1.5)(1.2,1)
\psline(-.8,1)(-.8,1.3)(-.2,1.3)(-.2,1)
\psline(.8,1)(.8,1.3)(.2,1.3)(.2,1)
\psline(-1.2,-1)(-1.2,-1.5)(1.2,-1.5)(1.2,-1)
\psline(-.8,-1)(-.8,-1.3)(-.2,-1.3)(-.2,-1)
\psline(.8,-1)(.8,-1.3)(.2,-1.3)(.2,-1)
\psline[arrowscale=1.5]{->}(-.1,1.5)(.1,1.5)
\psline[arrowscale=1.5]{->}(.1,-1.5)(-.1,-1.5)
\psline[arrowscale=1.5]{->}(.4,1.3)(.6,1.3)
\psline[arrowscale=1.5]{->}(.6,-1.3)(.4,-1.3)
\psline[arrowscale=1.5]{->}(-.6,1.3)(-.4,1.3)
\psline[arrowscale=1.5]{->}(-.4,-1.3)(-.6,-1.3)
\endpspicture \cong \pspicture[.45](-2.7,-2.4)(2.2,2.4)
\psline(-2,-2.2)(-2,-1.6)(-1.4,-1.3)(-1.64,-1.17)
\psline(-1.76,-1.12)(-2,-1)(-1.4,-.7) \psline(-1.74,-.82)(-2,-.7)
\psline(-1.76,-1.42)(-2,-1.3)(-1.4,-1)(-1.64,-.87)
\psline(-2,-.3)(-1.4,0)(-1.64,.13) \psline(-1.4,-.3)(-1.64,-.17)
\psline(-1.76,-.12)(-2,0)(-1.4,.3)(-1.64,.43)
\psline(-1.76,.18)(-2,.3)(-1.4,.6)(-1.64,.73)
\psline(-1.76,.48)(-2,.6)(-1.4,.9)
\psline(-1.76,.78)(-2,.9)(-1.4,1.2)(-1.64,1.33)
\psline[linecolor=darkred,linewidth=1pt](-1.76,1.38)(-2,1.5)(-2,2)(2,2)(2,1.5)(1.76,1.38)
\psline(-1.64,-1.47)(-1.4,-1.6)(-1.4,-1.8)(-.3,-1.8)
\psline(.3,1.5)(.3,1.7)(1.4,1.7)(1.4,1.5)(2,1.2)(1.76,1.08)
\psline(-1.4,.9)(-1.64,1.03) \psline(1.4,-1)(1.64,-1.12)
\psline(-1.74,1.08)(-2,1.2)(-1.4,1.5)(-1.4,1.7)(-.3,1.7)(-.3,1.5)(-.06,1.38)
\psline(.3,1.5)(-.3,1.2)(-.06,1.08)
\psline(.06,1.32)(.3,1.2)(-.3,.9)(-.06,.78)
\psline(.06,1.02)(.3,.9)(-.3,.6)(-.06,.48)
\psline(-.3,1.5)(-.06,1.38) \psline(.06,.72)(.3,.6)(-.3,.3)
\psline(.06,.42)(.3,.3) \psline(1.4,-1.6)(1.64,-1.48)
\psline(-2,-2.2)(2,-2.2)(2,-1.6)(1.4,-1.3)(1.64,-1.18)
\psline(1.76,-1.42)(2,-1.3)(1.4,-1)(1.64,-.88)
\psline(1.76,-1.12)(2,-1)(1.4,-.7)(1.64,-.58)
\psline(1.76,-.82)(2,-.7)(1.4,-.4)(1.64,-.28)
\psline(1.76,-.52)(2,-.4)(1.4,-.1) \psline(1.76,-.22)(2,-.1)
\psline(1.64,1.32)(1.4,1.2)(2,.9)(1.76,.78)
\psline(1.64,1.02)(1.4,.9)(2,.6)(1.76,.48)
\psline(1.64,.72)(1.4,.6)(2,.3) \psline(1.64,.42)(1.4,.3)
\psline(1.4,-1.6)(1.4,-1.8)(.3,-1.8)(.3,-1.6)
\psline(.3,-1.8)(.3,-1.6)(.06,-1.48)
\psline(-.3,-1.8)(-.3,-1.6)(.3,-1.3)(.06,-1.18)
\psline(-.06,-1.42)(-.3,-1.3)(.3,-1)(.06,-.88)
\psline(-.06,-1.12)(-.3,-1)(.3,-.7)(.06,-.58)
\psline(-.06,-.82)(-.3,-.7)(.3,-.4)(.06,-.28)
\psline(-.06,-.52)(-.3,-.4)(.3,-.1) \psline(-.06,-.22)(-.3,-.1)
\psline[arrowscale=1.5,linecolor=darkred]{->}(-.1,2)(.1,2)
\psline[arrowscale=1.5]{->}(.1,-2.2)(-.1,-2.2)
\psline[arrowscale=1.5]{->}(-.7,-1.8)(-.9,-1.8)
\psline[arrowscale=1.5]{->}(.9,-1.8)(.7,-1.8)
\psline[arrowscale=1.5]{->}(-.9,1.7)(-.7,1.7)
\psline[arrowscale=1.5]{->}(.7,1.7)(.9,1.7)
\rput[b](-1.7,-.45){$.$}\rput[b](-1.7,-.55){$.$}\rput[b](-1.7,-.65){$.$}
\rput[b](1.7,0){$.$}\rput[b](1.7,.1){$.$}\rput[b](1.7,.2){$.$}
\rput[b](0,0){$.$}\rput[b](0,.1){$.$}\rput[b](0,.2){$.$}
\endpspicture \cong \pspicture[.45](-2.7,-2.4)(2.2,2.4)
\pspolygon[linecolor=lightred,fillstyle=solid,fillcolor=lightred](-.29,1.69)(.29,1.69)(.29,1.51)(0,1.38)(-.29,1.51)
\pspolygon[linecolor=lightgreen,fillstyle=solid,fillcolor=lightgreen](-.99,-1.81)
(-.29,-1.81)(-.29,-1.51)(0,-1.38)(.29,-1.51)(.29,-1.81)(.99,-1.81)(.99,-1.99)(-.99,-1.99)
\pspolygon[linecolor=emgreen,fillstyle=solid,fillcolor=emgreen](-1.39,1.69)
(-.31,1.69)(-.31,1.49)(-.31,-1.81)(-.99,-1.81)(-.99,1.21)(-1.39,1.21)(-1.7,1.08)(-1.99,1.21)(-1.39,1.49)
\pspolygon[linecolor=emgreen,fillstyle=solid,fillcolor=emgreen](1.39,1.69)
(.31,1.69)(.31,1.49)(.31,-1.81)(.99,-1.81)(.99,1.21)(1.39,1.21)(1.7,1.08)(1.99,1.21)(1.39,1.49)
\pspolygon[linecolor=emgreen,fillstyle=solid,fillcolor=emgreen](-.29,1.49)(-.01,1.35)(-.29,1.21)
\pspolygon[linecolor=emgreen,fillstyle=solid,fillcolor=emgreen](-.29,1.19)(-.01,1.05)(-.29,.91)
\pspolygon[linecolor=emgreen,fillstyle=solid,fillcolor=emgreen](-.29,.89)(-.01,.75)(-.29,.61)
\pspolygon[linecolor=emgreen,fillstyle=solid,fillcolor=emgreen](-.29,.59)(-.01,.45)(-.29,.31)
\pspolygon[linecolor=emgreen,fillstyle=solid,fillcolor=emgreen](.29,1.49)(.01,1.35)(.29,1.21)
\pspolygon[linecolor=emgreen,fillstyle=solid,fillcolor=emgreen](.29,1.19)(.01,1.05)(.29,.91)
\pspolygon[linecolor=emgreen,fillstyle=solid,fillcolor=emgreen](.29,.89)(.01,.75)(.29,.61)
\pspolygon[linecolor=emgreen,fillstyle=solid,fillcolor=emgreen](.29,.59)(.01,.45)(.29,.31)
\pspolygon[linecolor=emgreen,fillstyle=solid,fillcolor=emgreen](-.29,-1.59)(-.01,-1.45)(-.29,-1.31)
\pspolygon[linecolor=emgreen,fillstyle=solid,fillcolor=emgreen](-.29,-1.29)(-.01,-1.15)(-.29,-1.01)
\pspolygon[linecolor=emgreen,fillstyle=solid,fillcolor=emgreen](-.29,-.99)(-.01,-.85)(-.29,-.71)
\pspolygon[linecolor=emgreen,fillstyle=solid,fillcolor=emgreen](-.29,-.69)(-.01,-.55)(-.29,-.41)
\pspolygon[linecolor=emgreen,fillstyle=solid,fillcolor=emgreen](-.29,-.39)(-.01,-.25)(-.29,-.11)
\pspolygon[linecolor=emgreen,fillstyle=solid,fillcolor=emgreen](.29,-1.59)(.01,-1.45)(.29,-1.31)
\pspolygon[linecolor=emgreen,fillstyle=solid,fillcolor=emgreen](.29,-1.29)(.01,-1.15)(.29,-1.01)
\pspolygon[linecolor=emgreen,fillstyle=solid,fillcolor=emgreen](.29,-.99)(.01,-.85)(.29,-.71)
\pspolygon[linecolor=emgreen,fillstyle=solid,fillcolor=emgreen](.29,-.69)(.01,-.55)(.29,-.41)
\pspolygon[linecolor=emgreen,fillstyle=solid,fillcolor=emgreen](.29,-.39)(.01,-.25)(.29,-.11)
\pspolygon[linecolor=emgreen,fillstyle=solid,fillcolor=emgreen](-1.99,-1.3)(-1.7,-1.16)(-1.41,-1.3)(-1.7,-1.44)
\pspolygon[linecolor=emgreen,fillstyle=solid,fillcolor=emgreen](1.99,-1.3)(1.7,-1.16)(1.41,-1.3)(1.7,-1.44)
\pspolygon[linecolor=emgreen,fillstyle=solid,fillcolor=emgreen](-1.99,-1)(-1.7,-.86)(-1.41,-1)(-1.7,-1.14)
\pspolygon[linecolor=emgreen,fillstyle=solid,fillcolor=emgreen](1.99,-1)(1.7,-.86)(1.41,-1)(1.7,-1.14)
\pspolygon[linecolor=emgreen,fillstyle=solid,fillcolor=emgreen](-1.99,-.7)(-1.41,-.7)(-1.7,-.84)
\pspolygon[linecolor=emgreen,fillstyle=solid,fillcolor=emgreen](1.99,-.7)(1.7,-.56)(1.41,-.7)(1.7,-.84)
\pspolygon[linecolor=emgreen,fillstyle=solid,fillcolor=emgreen](1.99,-.4)(1.7,-.26)(1.41,-.4)(1.7,-.54)
\pspolygon[linecolor=emgreen,fillstyle=solid,fillcolor=emgreen](1.99,-.1)(1.41,-.1)(1.7,-.24)
\pspolygon[linecolor=emgreen,fillstyle=solid,fillcolor=emgreen](-1.99,1.2)(-1.7,1.06)(-1.41,1.2)(-1.7,1.34)
\pspolygon[linecolor=emgreen,fillstyle=solid,fillcolor=emgreen](1.99,1.2)(1.7,1.06)(1.41,1.2)(1.7,1.34)
\pspolygon[linecolor=emgreen,fillstyle=solid,fillcolor=emgreen](-1.99,.9)(-1.7,.76)(-1.41,.9)(-1.7,1.04)
\pspolygon[linecolor=emgreen,fillstyle=solid,fillcolor=emgreen](1.99,.9)(1.7,.76)(1.41,.9)(1.7,1.04)
\pspolygon[linecolor=emgreen,fillstyle=solid,fillcolor=emgreen](-1.99,.6)(-1.7,.46)(-1.41,.6)(-1.7,.74)
\pspolygon[linecolor=emgreen,fillstyle=solid,fillcolor=emgreen](1.99,.6)(1.7,.46)(1.41,.6)(1.7,.74)
\pspolygon[linecolor=emgreen,fillstyle=solid,fillcolor=emgreen](-1.99,.3)(-1.7,.16)(-1.41,.3)(-1.7,.44)
\pspolygon[linecolor=emgreen,fillstyle=solid,fillcolor=emgreen](1.99,.3)(1.41,.3)(1.7,.44)
\pspolygon[linecolor=emgreen,fillstyle=solid,fillcolor=emgreen](-1.99,0)(-1.7,-.14)(-1.41,0)(-1.7,.14)
\pspolygon[linecolor=emgreen,fillstyle=solid,fillcolor=emgreen](-1.99,-.3)(-1.41,-.3)(-1.7,-.16)
\pspolygon[linecolor=emgreen,fillstyle=solid,fillcolor=emgreen](-1.99,-2.19)(-1.99,-1.61)(-1.7,-1.46)(-1.41,-1.61)(-1.41,-1.81)
(-1.01,-1.81)(-1.01,-2.01)(1.01,-2.01)(1.01,-1.81)(1.41,-1.81)(1.41,-1.61)(1.7,-1.46)(1.99,-1.61)(1.99,-2.19)
\pspolygon[linecolor=lightred,fillstyle=solid,fillcolor=lightred](-1.99,-1.59)(-1.71,-1.45)(-1.99,-1.31)
\pspolygon[linecolor=lightred,fillstyle=solid,fillcolor=lightred](-1.99,-1.29)(-1.71,-1.15)(-1.99,-1.01)
\pspolygon[linecolor=lightred,fillstyle=solid,fillcolor=lightred](-1.99,-.99)(-1.71,-.85)(-1.99,-.71)
\pspolygon[linecolor=lightred,fillstyle=solid,fillcolor=lightred](-1.99,1.19)(-1.71,1.05)(-1.99,.91)
\pspolygon[linecolor=lightred,fillstyle=solid,fillcolor=lightred](-1.99,.89)(-1.71,.75)(-1.99,.61)
\pspolygon[linecolor=lightred,fillstyle=solid,fillcolor=lightred](-1.99,.59)(-1.71,.45)(-1.99,.31)
\pspolygon[linecolor=lightred,fillstyle=solid,fillcolor=lightred](-1.99,.29)(-1.71,.15)(-1.99,.01)
\pspolygon[linecolor=lightred,fillstyle=solid,fillcolor=lightred](-1.99,-.01)(-1.71,-.15)(-1.99,-.29)
\pspolygon[linecolor=lightred,fillstyle=solid,fillcolor=lightred](-1.41,-1.59)(-1.69,-1.45)(-1.41,-1.31)
\pspolygon[linecolor=lightred,fillstyle=solid,fillcolor=lightred](-1.41,-1.29)(-1.69,-1.15)(-1.41,-1.01)
\pspolygon[linecolor=lightred,fillstyle=solid,fillcolor=lightred](-1.41,-.99)(-1.69,-.85)(-1.41,-.71)
\pspolygon[linecolor=lightred,fillstyle=solid,fillcolor=lightred](-1.41,1.19)(-1.69,1.05)(-1.41,.91)
\pspolygon[linecolor=lightred,fillstyle=solid,fillcolor=lightred](-1.41,.89)(-1.69,.75)(-1.41,.61)
\pspolygon[linecolor=lightred,fillstyle=solid,fillcolor=lightred](-1.41,.59)(-1.69,.45)(-1.41,.31)
\pspolygon[linecolor=lightred,fillstyle=solid,fillcolor=lightred](-1.41,.29)(-1.69,.15)(-1.41,.01)
\pspolygon[linecolor=lightred,fillstyle=solid,fillcolor=lightred](-1.41,-.01)(-1.69,-.15)(-1.41,-.29)
\pspolygon[linecolor=lightred,fillstyle=solid,fillcolor=lightred](1.99,-1.59)(1.71,-1.45)(1.99,-1.31)
\pspolygon[linecolor=lightred,fillstyle=solid,fillcolor=lightred](1.99,-1.29)(1.71,-1.15)(1.99,-1.01)
\pspolygon[linecolor=lightred,fillstyle=solid,fillcolor=lightred](1.99,-.99)(1.71,-.85)(1.99,-.71)
\pspolygon[linecolor=lightred,fillstyle=solid,fillcolor=lightred](1.99,-.69)(1.71,-.55)(1.99,-.41)
\pspolygon[linecolor=lightred,fillstyle=solid,fillcolor=lightred](1.99,-.39)(1.71,-.25)(1.99,-.11)
\pspolygon[linecolor=lightred,fillstyle=solid,fillcolor=lightred](1.99,1.19)(1.71,1.05)(1.99,.91)
\pspolygon[linecolor=lightred,fillstyle=solid,fillcolor=lightred](1.99,.89)(1.71,.75)(1.99,.61)
\pspolygon[linecolor=lightred,fillstyle=solid,fillcolor=lightred](1.99,.59)(1.71,.45)(1.99,.31)
\pspolygon[linecolor=lightred,fillstyle=solid,fillcolor=lightred](1.41,-1.59)(1.69,-1.45)(1.41,-1.31)
\pspolygon[linecolor=lightred,fillstyle=solid,fillcolor=lightred](1.41,-1.29)(1.69,-1.15)(1.41,-1.01)
\pspolygon[linecolor=lightred,fillstyle=solid,fillcolor=lightred](1.41,-.99)(1.69,-.85)(1.41,-.71)
\pspolygon[linecolor=lightred,fillstyle=solid,fillcolor=lightred](1.41,-.69)(1.69,-.55)(1.41,-.41)
\pspolygon[linecolor=lightred,fillstyle=solid,fillcolor=lightred](1.41,-.39)(1.69,-.25)(1.41,-.11)
\pspolygon[linecolor=lightred,fillstyle=solid,fillcolor=lightred](1.41,1.19)(1.69,1.05)(1.41,.91)
\pspolygon[linecolor=lightred,fillstyle=solid,fillcolor=lightred](1.41,.89)(1.69,.75)(1.41,.61)
\pspolygon[linecolor=lightred,fillstyle=solid,fillcolor=lightred](1.41,.59)(1.69,.45)(1.41,.31)
\pspolygon[linecolor=lightred,fillstyle=solid,fillcolor=lightred](1.41,-1.79)(1.01,-1.79)(1.01,1.19)(1.41,1.19)
\pspolygon[linecolor=lightred,fillstyle=solid,fillcolor=lightred](-1.41,-1.79)(-1.01,-1.79)(-1.01,1.19)(-1.41,1.19)
\pspolygon[linecolor=lightred,fillstyle=solid,fillcolor=lightred](-.29,-1.3)(0,-1.16)(.29,-1.3)(0,-1.44)
\pspolygon[linecolor=lightgreen,fillstyle=solid,fillcolor=lightgreen](-.29,-1)(0,-.86)(.29,-1)(0,-1.14)
\pspolygon[linecolor=lightred,fillstyle=solid,fillcolor=lightred](-.29,-.7)(0,-.56)(.29,-.7)(0,-.84)
\pspolygon[linecolor=lightgreen,fillstyle=solid,fillcolor=lightgreen](-.29,-.4)(0,-.26)(.29,-.4)(0,-.54)
\pspolygon[linecolor=lightred,fillstyle=solid,fillcolor=lightred](-.29,-.1)(.29,-.1)(0,-.24)
\pspolygon[linecolor=lightgreen,fillstyle=solid,fillcolor=lightgreen](-.29,1.2)(0,1.06)(.29,1.2)(0,1.34)
\pspolygon[linecolor=lightred,fillstyle=solid,fillcolor=lightred](-.29,.9)(0,.76)(.29,.9)(0,1.04)
\pspolygon[linecolor=lightgreen,fillstyle=solid,fillcolor=lightgreen](-.29,.6)(0,.46)(.29,.6)(0,.74)
\pspolygon[linecolor=lightred,fillstyle=solid,fillcolor=lightred](-.29,.3)(.29,.3)(0,.44)
\psline(-2,-2.2)(-2,-1.6)(-1.4,-1.3)(-1.64,-1.17)
\psline(-1.76,-1.12)(-2,-1)(-1.4,-.7) \psline(-1.74,-.82)(-2,-.7)
\psline(-1.76,-1.42)(-2,-1.3)(-1.4,-1)(-1.64,-.87)
\psline(-2,-.3)(-1.4,0)(-1.64,.13) \psline(-1.4,-.3)(-1.64,-.17)
\psline(-1.76,-.12)(-2,0)(-1.4,.3)(-1.64,.43)
\psline(-1.76,.18)(-2,.3)(-1.4,.6)(-1.64,.73)
\psline(-1.76,.48)(-2,.6)(-1.4,.9)
\psline(-1.76,.78)(-2,.9)(-1.4,1.2)
\psline(-1.64,-1.47)(-1.4,-1.6)(-1.4,-1.8)(-.3,-1.8)
\psline(.3,1.5)(.3,1.7)(1.4,1.7)(1.4,1.5)(2,1.2)(1.76,1.08)
\psline(-1.4,.9)(-1.64,1.03) \psline(1.4,-1)(1.64,-1.12)
\psline(-1.74,1.08)(-2,1.2)(-1.4,1.5)(-1.4,1.7)(-.3,1.7)(-.3,1.5)(-.06,1.38)
\psline(.3,1.5)(-.3,1.2)(-.06,1.08)
\psline(.06,1.32)(.3,1.2)(-.3,.9)(-.06,.78)
\psline(.06,1.02)(.3,.9)(-.3,.6)(-.06,.48)
\psline(-.3,1.5)(-.06,1.38) \psline(.06,.72)(.3,.6)(-.3,.3)
\psline(.06,.42)(.3,.3) \psline(1.4,-1.6)(1.64,-1.48)
\psline(-2,-2.2)(2,-2.2)(2,-1.6)(1.4,-1.3)(1.64,-1.18)
\psline(1.76,-1.42)(2,-1.3)(1.4,-1)(1.64,-.88)
\psline(1.76,-1.12)(2,-1)(1.4,-.7)(1.64,-.58)
\psline(1.76,-.82)(2,-.7)(1.4,-.4)(1.64,-.28)
\psline(1.76,-.52)(2,-.4)(1.4,-.1) \psline(1.76,-.22)(2,-.1)
\psline(1.4,1.2)(2,.9)(1.76,.78)
\psline(1.64,1.02)(1.4,.9)(2,.6)(1.76,.48)
\psline(1.64,.72)(1.4,.6)(2,.3) \psline(1.64,.42)(1.4,.3)
\psline(1.4,-1.6)(1.4,-1.8)(.3,-1.8)(.3,-1.6)
\psline(.3,-1.8)(.3,-1.6)(.06,-1.48)
\psline(-.3,-1.8)(-.3,-1.6)(.3,-1.3)(.06,-1.18)
\psline(-.06,-1.42)(-.3,-1.3)(.3,-1)(.06,-.88)
\psline(-.06,-1.12)(-.3,-1)(.3,-.7)(.06,-.58)
\psline(-.06,-.82)(-.3,-.7)(.3,-.4)(.06,-.28)
\psline(-.06,-.52)(-.3,-.4)(.3,-.1) \psline(-.06,-.22)(-.3,-.1)
\psline[arrowscale=1.5,linecolor=darkred]{->}(-.1,-2)(.1,-2)
\psline[arrowscale=1.5]{->}(.1,-2.2)(-.1,-2.2)
\psline[arrowscale=1.5]{->}(-.7,-1.8)(-.9,-1.8)
\psline[arrowscale=1.5]{->}(.9,-1.8)(.7,-1.8)
\psline[arrowscale=1.5]{->}(-.9,1.7)(-.7,1.7)
\psline[arrowscale=1.5]{->}(.7,1.7)(.9,1.7)
\psline[linecolor=darkred,linewidth=1.2pt](-1.4,1.2)(-1,1.2)(-1,-1.75)
\psline[linecolor=darkred,linewidth=1.2pt](-1,-1.85)(-1,-2)(1,-2)(1,-1.85)
\psline[linecolor=darkred,linewidth=1.2pt](1,-1.75)(1,1.2)(1.4,1.2)
\rput[b](-1.7,-.45){$.$}\rput[b](-1.7,-.55){$.$}\rput[b](-1.7,-.65){$.$}
\rput[b](1.7,0){$.$}\rput[b](1.7,.1){$.$}\rput[b](1.7,.2){$.$}
\rput[b](0,0){$.$}\rput[b](0,.1){$.$}\rput[b](0,.2){$.$}
\endpspicture
$$
\caption{Minimal genus Seifert surfaces of the pretzel knots
$K(p,q,r)$.} \label{fig4}
\end{figure}

\begin{thm}
The surface obtained by applying Seifert's algorithm to the
pretzel knot $K(p, q, r)$ as in Figure~\ref{fig4} is a minimal genus
Seifert surface, if $1/|p| + 1/|q| + 1/|r| \le 1$.
\label{Seifertthm} \end{thm}
\begin{proof}
We consider two cases : $i)$ all of $p, q, r$ are odd, $ii)$ exactly
one of $p, q, r$ is even. For the first case, the first Seifert
surface in Figure~\ref{fig4} is clearly a minimal genus since its
genus is $1$ unless $K(p, q, r)$ is the unknot. But it can not be
the unknot by the hypothesis. For the second case, we can consider
$K(-2l, q, \pm r)$, $K(-2l, q, \pm r)$ or their mirror images, where $l, q,
r$ are positive. Their canonical Seifert surfaces are given in
Figure~\ref{fig4}. To prove these surfaces are minimal genus Seifert surfaces, first
we find $2 g(K(-2l,q,\pm r)) \ge q + r - 1 \pm 1$ using the
Alexander polynomials of $K(-2l, q, r)$ and $K(-2l, q, -r)$ given in
Lemma~\ref{alexanderlem} and inequality (\ref{genineqal}). But the
genus of the second Seifert surface in Figure~\ref{fig4} is
$(q+r)/2$, and the third surface in Figure~\ref{fig4} is
$(q+r-2)/2$. It completes the proof.
\end{proof}

By combining Theorem~\ref{toruspretzel} and
Theorem~\ref{Seifertthm}, we find the following corollary.

\begin{cor}  The genus of $K(p,q,r)$ is as follows.
\item{$\mathrm{1)}$} $K(p, \pm 1,\mp 1), K(\pm 2, \mp 1, \pm 3)$
have genus 0 for all $p$.
\item{$\mathrm{2)}$} $K(p, q , r)$ has genus 1 if
$p \equiv q \equiv r \equiv 1 ~(mod~2)$ and we are not in case 1).
\item{$\mathrm{3)}$} $K(\pm 2, \mp 1, \pm r)$ has genus $(|r-2|-1)/2$.
\item{$\mathrm{4)}$} $K(\mp 2l, q, r)$ has genus $(|q| + |r|)/2$ if $q, r$
have the same sign and we are not in any of the previous cases.
\item{$\mathrm{5)}$} $K(\mp 2l, q, r)$ has genus $(|q| + |r| -2)/2$ if
$q, r$ have different signs and we are not in cases $1), 2)$ or
$3)$.
\end{cor}

For classical pretzel links, one can see that $L(2l_1,2l_2,2l_3)$
has genus $0$. For $L(2l_1, 2l_2$, $r)$, we are going to see more
interesting results for the genus because there is a freedom to
choose orientations of the components. But, Lemma~\ref{alexanderlem}
remains true for arbitrary integers $q, r$, so we can find the
following corollary.
\begin{cor}
The genus of the link $L(2l_1,2l_2,r)$, where $|l_1|\ge |l_2|$ ,
$l_1, l_2 >0$(unless we indicate differently) and $r \ge 0$, is as follows.
\item{$\mathrm{1)}$} $L(2l_1, 2l_2 , \pm r)$ has genus 0
if $r \equiv 0 ~(mod~2)$ and $l_1, l_2$ are nonzero integers.
\item{$\mathrm{2)}$} $L(\pm 2, \pm 2l_2 , \mp 1)$ has genus $(|2l_2-2|-2)/2$.
\item{$\mathrm{3)}$} $K(\mp 2l_1, \mp 2l_2, \mp r)$ has genus $(|l_2| + |r|-1)/2$
if we are not in one of the previous cases.
\item{$\mathrm{4)}$} $K(\mp 2l_1, \mp 2l_2, \pm r)$ has genus $(|l_2| + |r|-3)/2$
if we are not in any of the previous cases.
\item{$\mathrm{5)}$} $K(\mp 2l_1, \pm 2l_2, \mp r)$ has genus $(|l_2| + |r| -3)/2$
if we are  not in case $1)$.
\item{$\mathrm{6)}$} $K(\mp 2l_1, \pm 2l_2, \pm r)$ has genus $(|l_2| + |r| -1)/2$
if we are not in case $1)$ and $|l_1|>|l_2|$, or has genus $(|l_2| +
|r| -3)/2$  if we are  not in case $1)$ and $|l_1|=|l_2|$.
\label{3pretzelgen}
\end{cor}
\begin{proof}
We follow the proof of Theorem~\ref{toruspretzel} and
Theorem~\ref{Seifertthm} carefully ; if  $r=\pm 1$, the link will
have two representatives by the move we used in the proof of
Theorem~\ref{alterthm}, we get the result, with a note that we
have a freedom to choose an orientation of the component which goes
through two even crossing boxes.
\end{proof}

\section{Conway polynomials of $n$-pretzel links} \label{poly}

To find the polynomial invariants of $n$-pretzel links, we will use
a computation tree : a \emph{computation tree} of a link polynomial
$P_L$ is an edge weighted, rooted binary tree whose vertices are
links, the root of the tree is $L$, two vertices $L_1, L_2$
are children of a vertex $L_p$ if
$$P_{L_p}=w(L_{p(1)})P_{L_1}+ w(L_{p(2)})P_{L_2},$$
and $w(L_{p(i)})$ is the weight on the edge between $L_p$ and $L_i$.
One can see that the link polynomial $P_L$ can be computed as
follows,
$$P_{L}=\sum_{L_v\in \mathcal{L}}
\prod_{L_p\in \mathcal{P}(L_v)} w(L_{p(i)})P_{L_v},$$ where
$\mathcal{L}$ is the set of all vertices of valence $1$ and
$\mathcal{P}(L_v)$ is the set of all vertices of the path from the
root to the vertex $L_v$. In general, it is easy to find $P_L$ if we
repeatedly use the skein relations until each vertex $L_v$ becomes an unlink. Moreover, one can replace
links by other for a convenience of the computation. For instance,
J. Franks and R. F. Williams used braids to find a beautiful result on Jones polynomial~\cite{FW:jones}.

\begin{figure}
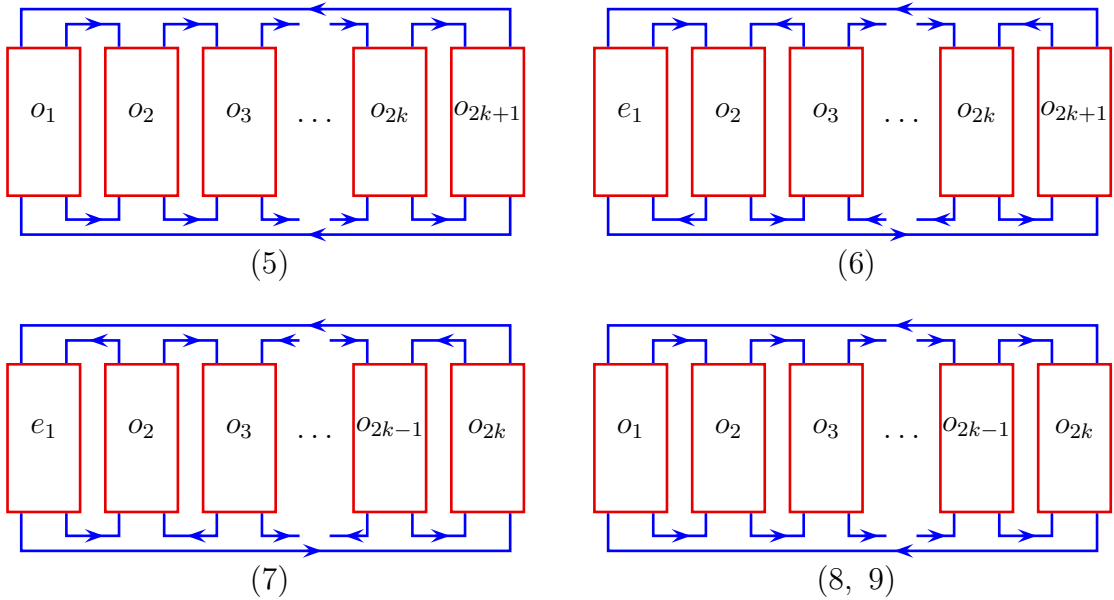

$$
\pspicture[.4](-4.3,-2.2)(3,1.8)
\rput(-.6,-1.9){\rnode{b1}{$(\ref{conknot3})$}}
\psframe[linecolor=darkred](-4.1,-1)(-3.1,1) \rput[b](-3.6,0){$o_1$}
\psframe[linecolor=darkred](-2.8,-1)(-1.8,1) \rput[b](-2.3,0){$o_2$}
\psframe[linecolor=darkred](-1.5,-1)(-.5,1) \rput[b](-1,0){$o_3$}
\rput[b](0,0){$\ldots$} \psframe[linecolor=darkred](.5,-1)(1.5,1)
\rput[b](1,0){$o_{2k}$} \psframe[linecolor=darkred](1.8,-1)(2.8,1)
\rput[b](2.3,0){$o_{2k+1}$}
\psline(-3.9,1)(-3.9,1.5)(2.6,1.5)(2.6,1)
\psline(-3.3,1)(-3.3,1.3)(-2.6,1.3)(-2.6,1)
\psline(-2,1)(-2,1.3)(-1.3,1.3)(-1.3,1)
\psline(-.7,1)(-.7,1.3)(-.2,1.3) \psline(.2,1.3)(.7,1.3)(.7,1)
\psline(1.3,1)(1.3,1.3)(2,1.3)(2,1)
\psline(-3.9,-1)(-3.9,-1.5)(2.6,-1.5)(2.6,-1)
\psline(-3.3,-1)(-3.3,-1.3)(-2.6,-1.3)(-2.6,-1)
\psline(-2,-1)(-2,-1.3)(-1.3,-1.3)(-1.3,-1)
\psline(-.7,-1)(-.7,-1.3)(-.2,-1.3) \psline(.2,-1.3)(.7,-1.3)(.7,-1)
\psline(1.3,-1)(1.3,-1.3)(2,-1.3)(2,-1)
\psline[arrowscale=1.5]{->}(.1,1.5)(-.1,1.5)
\psline[arrowscale=1.5]{->}(-3,1.3)(-2.8,1.3)
\psline[arrowscale=1.5]{->}(-1.7,1.3)(-1.5,1.3)
\psline[arrowscale=1.5]{->}(-.5,1.3)(-.3,1.3)
\psline[arrowscale=1.5]{->}(.4,1.3)(.6,1.3)
\psline[arrowscale=1.5]{->}(1.6,1.3)(1.8,1.3)
\psline[arrowscale=1.5]{->}(.1,-1.5)(-.1,-1.5)
\psline[arrowscale=1.5]{->}(-3,-1.3)(-2.8,-1.3)
\psline[arrowscale=1.5]{->}(-1.7,-1.3)(-1.5,-1.3)
\psline[arrowscale=1.5]{->}(-.5,-1.3)(-.3,-1.3)
\psline[arrowscale=1.5]{->}(.4,-1.3)(.6,-1.3)
\psline[arrowscale=1.5]{->}(1.6,-1.3)(1.8,-1.3)
\endpspicture \hskip .5cm \pspicture[.4](-4.3,-2.2)(3,1.8) \rput(-.6,-1.9){\rnode{b1}{$(\ref{conknot4})$}}
\psframe[linecolor=darkred](-4.1,-1)(-3.1,1) \rput[b](-3.6,0){$e_1$}
\psframe[linecolor=darkred](-2.8,-1)(-1.8,1) \rput[b](-2.3,0){$o_2$}
\psframe[linecolor=darkred](-1.5,-1)(-.5,1) \rput[b](-1,0){$o_3$}
\rput[b](0,0){$\ldots$} \psframe[linecolor=darkred](.5,-1)(1.5,1)
\rput[b](1,0){$o_{2k}$} \psframe[linecolor=darkred](1.8,-1)(2.8,1)
\rput[b](2.3,0){$o_{2k+1}$}
\psline(-3.9,1)(-3.9,1.5)(2.6,1.5)(2.6,1)
\psline(-3.3,1)(-3.3,1.3)(-2.6,1.3)(-2.6,1)
\psline(-2,1)(-2,1.3)(-1.3,1.3)(-1.3,1)
\psline(-.7,1)(-.7,1.3)(-.2,1.3) \psline(.2,1.3)(.7,1.3)(.7,1)
\psline(1.3,1)(1.3,1.3)(2,1.3)(2,1)
\psline(-3.9,-1)(-3.9,-1.5)(2.6,-1.5)(2.6,-1)
\psline(-3.3,-1)(-3.3,-1.3)(-2.6,-1.3)(-2.6,-1)
\psline(-2,-1)(-2,-1.3)(-1.3,-1.3)(-1.3,-1)
\psline(-.7,-1)(-.7,-1.3)(-.2,-1.3) \psline(.2,-1.3)(.7,-1.3)(.7,-1)
\psline(1.3,-1)(1.3,-1.3)(2,-1.3)(2,-1)
\psline[arrowscale=1.5]{->}(.1,1.5)(-.1,1.5)
\psline[arrowscale=1.5]{->}(-3,1.3)(-2.8,1.3)
\psline[arrowscale=1.5]{<-}(-1.7,1.3)(-1.5,1.3)
\psline[arrowscale=1.5]{->}(-.5,1.3)(-.3,1.3)
\psline[arrowscale=1.5]{->}(.4,1.3)(.6,1.3)
\psline[arrowscale=1.5]{<-}(1.6,1.3)(1.8,1.3)
\psline[arrowscale=1.5]{<-}(.1,-1.5)(-.1,-1.5)
\psline[arrowscale=1.5]{<-}(-3,-1.3)(-2.8,-1.3)
\psline[arrowscale=1.5]{->}(-1.7,-1.3)(-1.5,-1.3)
\psline[arrowscale=1.5]{<-}(-.5,-1.3)(-.3,-1.3)
\psline[arrowscale=1.5]{<-}(.4,-1.3)(.6,-1.3)
\psline[arrowscale=1.5]{->}(1.6,-1.3)(1.8,-1.3)
\endpspicture
$$
$$
\pspicture[.4](-4.3,-2.2)(3,1.8)
\rput(-.6,-1.9){\rnode{b1}{$(\ref{conknot5})$}}
\psframe[linecolor=darkred](-4.1,-1)(-3.1,1) \rput[b](-3.6,0){$e_1$}
\psframe[linecolor=darkred](-2.8,-1)(-1.8,1) \rput[b](-2.3,0){$o_2$}
\psframe[linecolor=darkred](-1.5,-1)(-.5,1) \rput[b](-1,0){$o_3$}
\rput[b](0,0){$\ldots$} \psframe[linecolor=darkred](.5,-1)(1.5,1)
\rput[b](1,0){$o_{2k-1}$} \psframe[linecolor=darkred](1.8,-1)(2.8,1)
\rput[b](2.3,0){$o_{2k}$}  \psline(-3.9,1)(-3.9,1.5)(2.6,1.5)(2.6,1)
\psline(-3.3,1)(-3.3,1.3)(-2.6,1.3)(-2.6,1)
\psline(-2,1)(-2,1.3)(-1.3,1.3)(-1.3,1)
\psline(-.7,1)(-.7,1.3)(-.2,1.3) \psline(.2,1.3)(.7,1.3)(.7,1)
\psline(1.3,1)(1.3,1.3)(2,1.3)(2,1)
\psline(-3.9,-1)(-3.9,-1.5)(2.6,-1.5)(2.6,-1)
\psline(-3.3,-1)(-3.3,-1.3)(-2.6,-1.3)(-2.6,-1)
\psline(-2,-1)(-2,-1.3)(-1.3,-1.3)(-1.3,-1)
\psline(-.7,-1)(-.7,-1.3)(-.2,-1.3) \psline(.2,-1.3)(.7,-1.3)(.7,-1)
\psline(1.3,-1)(1.3,-1.3)(2,-1.3)(2,-1)
\psline[arrowscale=1.5]{->}(.1,1.5)(-.1,1.5)
\psline[arrowscale=1.5]{<-}(-3,1.3)(-2.8,1.3)
\psline[arrowscale=1.5]{->}(-1.7,1.3)(-1.5,1.3)
\psline[arrowscale=1.5]{<-}(-.5,1.3)(-.3,1.3)
\psline[arrowscale=1.5]{->}(.4,1.3)(.6,1.3)
\psline[arrowscale=1.5]{<-}(1.6,1.3)(1.8,1.3)
\psline[arrowscale=1.5]{<-}(.1,-1.5)(-.1,-1.5)
\psline[arrowscale=1.5]{->}(-3,-1.3)(-2.8,-1.3)
\psline[arrowscale=1.5]{<-}(-1.7,-1.3)(-1.5,-1.3)
\psline[arrowscale=1.5]{->}(-.5,-1.3)(-.3,-1.3)
\psline[arrowscale=1.5]{<-}(.4,-1.3)(.6,-1.3)
\psline[arrowscale=1.5]{->}(1.6,-1.3)(1.8,-1.3)
\endpspicture \hskip .5cm \pspicture[.4](-4.3,-2.2)(3,1.8) \rput(-.6,-1.9){\rnode{b1}{$(\ref{conknot1}, ~\ref{conknot2})$}}
\psframe[linecolor=darkred](-4.1,-1)(-3.1,1) \rput[b](-3.6,0){$o_1$}
\psframe[linecolor=darkred](-2.8,-1)(-1.8,1) \rput[b](-2.3,0){$o_2$}
\psframe[linecolor=darkred](-1.5,-1)(-.5,1) \rput[b](-1,0){$o_3$}
\rput[b](0,0){$\ldots$} \psframe[linecolor=darkred](.5,-1)(1.5,1)
\rput[b](1,0){$o_{2k-1}$} \psframe[linecolor=darkred](1.8,-1)(2.8,1)
\rput[b](2.3,0){$o_{2k}$}  \psline(-3.9,1)(-3.9,1.5)(2.6,1.5)(2.6,1)
\psline(-3.3,1)(-3.3,1.3)(-2.6,1.3)(-2.6,1)
\psline(-2,1)(-2,1.3)(-1.3,1.3)(-1.3,1)
\psline(-.7,1)(-.7,1.3)(-.2,1.3) \psline(.2,1.3)(.7,1.3)(.7,1)
\psline(1.3,1)(1.3,1.3)(2,1.3)(2,1)
\psline(-3.9,-1)(-3.9,-1.5)(2.6,-1.5)(2.6,-1)
\psline(-3.3,-1)(-3.3,-1.3)(-2.6,-1.3)(-2.6,-1)
\psline(-2,-1)(-2,-1.3)(-1.3,-1.3)(-1.3,-1)
\psline(-.7,-1)(-.7,-1.3)(-.2,-1.3) \psline(.2,-1.3)(.7,-1.3)(.7,-1)
\psline(1.3,-1)(1.3,-1.3)(2,-1.3)(2,-1)
\psline[arrowscale=1.5]{->}(.1,1.5)(-.1,1.5)
\psline[arrowscale=1.5]{->}(-3,1.3)(-2.8,1.3)
\psline[arrowscale=1.5]{->}(-1.7,1.3)(-1.5,1.3)
\psline[arrowscale=1.5]{->}(-.5,1.3)(-.3,1.3)
\psline[arrowscale=1.5]{->}(.4,1.3)(.6,1.3)
\psline[arrowscale=1.5]{->}(1.6,1.3)(1.8,1.3)
\psline[arrowscale=1.5]{->}(.1,-1.5)(-.1,-1.5)
\psline[arrowscale=1.5]{->}(-3,-1.3)(-2.8,-1.3)
\psline[arrowscale=1.5]{->}(-1.7,-1.3)(-1.5,-1.3)
\psline[arrowscale=1.5]{->}(-.5,-1.3)(-.3,-1.3)
\psline[arrowscale=1.5]{->}(.4,-1.3)(.6,-1.3)
\psline[arrowscale=1.5]{->}(1.6,-1.3)(1.8,-1.3)
\endpspicture
$$
\caption{all oriented $n$-pretzel knots $L(p_1, p_2, \ldots , p_n)$}
\label{allnpretzel}
\end{figure}

To compute Conway Polynomials of $n$-pretzel links, we will use a new
notation for $n$-pretzel links which will be used for vertices of a
computation tree. We called a rectangle in Figure~\ref{allnpretzel}
\emph{a box} and the link moves in the \emph{same direction} in a
box if it has the orientation as in the second box from the left of the diagram
$(\ref{conknot4})$ of Figure~\ref{allnpretzel}, in the \emph{opposite directions}
if it has the orientation as in the first box from the left of the diagram
$(\ref{conknot4})$ in Figure~\ref{allnpretzel}. If we have a box for which two
strings move in the opposite directions and we use the skein relation at
this box, then the resulting links have either less number of the
boxes or less number of crossings. One can see that an opposite
direction can be happened only for a box with even number of
crossings (but this is not sufficient) except in the case that $n$
is even and all the $p_i$'s are odd (we will handle this case
separately). Suppose we have at least one even crossing box. We may
assume that it is $p_1=2l_1$. Let us remark that the Conway
polynomial vanishes for split links. The following is our new
notation for $n$-pretzel links. From a given $n$-pretzel link $L$
with an orientation $O$, we can represent $L$ by a vector in
$(\mathbb{Z}\times\mathbb{Z}_2)^n$ such as $(p_1^{\epsilon
_1},p_2^{\epsilon _2}, \ldots,p_n^{\epsilon _n})$, where $\epsilon
_i = 1(-1)$ if the link moves in the same(opposite, respectively)
direction in the box corresponding to $p_i$ with respect to the
given orientation $O$. Write $p_i^{1}=p_i$. First we find the
following recursive formula,

\begin{align*}
\cp_{L(p_1^{\epsilon _1},p_2^{\epsilon _2},\ldots ,p_i^{-1},\ldots
,p_n^{\epsilon _n})} &= \cp _{T_{(2,p_1^{\epsilon _1})}} \cp
_{T_{(2,p_2^{\epsilon _1})}} \ldots \hat{\cp _{T_{(2,p_i^{-1})}}}
\ldots \cp _{T_{(2,p_n^{\epsilon _n})}}\\ & -l_i z \cp
_{L(p_1^{\epsilon _1}, p_2^{\epsilon _2}, \ldots , \hat{p_i^{-1}},
\ldots , p_n^{\epsilon _n})},
\end{align*}
where the term under $\hat{~}$ is deleted.

By repeatedly using above formulae, we can make a computation
tree that there is no negative $\epsilon _i$ for the
representative at each vertex of valence $1$.
Then, we can expand $(\ldots, p_i,\ldots )$
into $(\ldots, p_i \pm 1(=p_i'),\ldots )$ and $(\ldots, p_i \pm 2,
\ldots )$ with suitable weights on edges, $1$ or $\pm z$ where
$|p_i|>|p_i'|$. We can keep on expanding at the crossings until all
the entries in the vectors  of vertices of valence $1$ are either $0$ or $\pm 1$. At this
stage, if it has more than two $0$'s then we stop the expansion and
change the vertex to zero because it is a split link. If it has only
one zero, it is a composite link of $T_{(2,p_i)}$'s. Otherwise, we
change the vector to an integral value $m$, the sum of the signs of
entries in the vector. In fact, it is the closed braid of two
strings represented by $\sigma_1^m$. Therefore, we can compute the
Conway polynomial of a link $L$ using this computation tree and the
Conway polynomial of closed $2$-braids.

\subsection{Conway polynomial of $n$-pretzel knots}

The general figures of $n$-pretzel knots are given in
Figure~\ref{allnpretzel} (the right-top one is a two components
link) where $e_1= 2l, o_i =2k_i+1$. We can see that there is at most
one box in which the knot moves in opposite directions. But for a
two component link, all boxes might move in opposite directions for
the orientation which is not in Figure~\ref{allnpretzel}.
Counterclockwise from the top-right, we get representatives,
$(o_1^{-1},o_2^{-1},\ldots , o_{2k}^{-1})$, $(o_1,o_2,\ldots
,o_{2k})$, $(o_1^{-1},o_2^{-1},$ $\ldots ,$ $ o_{2k+1}^{-1})$,
$(e_1^{-1}$ $, o_2,$ $ o_3,$ $\ldots ,$ $o_{2k+1})$ and $(e_1, o_2
,o_3, \ldots ,o_{2k})$. By using a computation tree for these
representatives, we find Theorem~\ref{npretzelconway}. For
convenience, we abbreviate $\cp _{T_{(2,n)}}$ by $\cp _{n}$
throughout the section.

\begin{thm}
 Let $e_1'=sign(e_1)(|e_1|-1)$, $o_i'=sign(o_i)(|o_i|-1)$,
$\alpha = \sum_{i=2}^{n} sign(o_i)$ and $\beta = sign(e_1)$. The
Conway polynomials of $n$-pretzel knots in Figure~\ref{allnpretzel}
are

\begin{align}
\cp _{L(o_1,o_2,o_3,\ldots,o_n)} &= \sum_{i=0}^{(n-1)/2} a_{i}
z^{2i}, \label{conknot3}\\ \cp_{L(e_1,o_2,o_3,\ldots,o_n)} &= \cp
_{o_2}\cp _{o_3}\ldots\cp _{o_n} [1-lz[-\frac{\alpha}{2} z +
\sum_{i=2}^{n} \frac{\cp _{o_{i}'}}{\cp _{o_{i}}}]],
\label{conknot4}
\\ \cp _{L(e_1,o_2,o_3,\ldots,o_n)} &= \cp _{o_2}\cp _{o_3}\ldots\cp
_{o_n} [\cp_{e_1'} + \cp_{e_1}[-\frac{\beta +\alpha}{2}z +
\sum_{i=2}^{n} \frac{\cp _{o_{i}'}}{\cp _{o_{i}}}]],\label{conknot5}
\\
 \cp _{L(o_1,o_2,o_3,\ldots,o_n)} &= \sum_{i=1}^{(n+1)/2} a_{i}
z^{2i-1},\label{conknot1} \\
 \cp _{L(o_1,o_2,o_3,\ldots,o_n)} &=  \cp _{o_1}\cp
_{o_2}\ldots\cp _{o_n} [\cp _{\sum_{i=1}^{n} sign(o_{i})} +
\sum_{i=1}^{n} \frac{\cp _{o_{i}'}}{\cp _{o_{i}}}],\label{conknot2}
\end{align}
where for $L(o_1,o_2,o_3,\ldots,o_n)$ we have two possible
orientations because it is a two components link, so we get
$\mathrm{(\ref{conknot1})}$ for $(o_1^{-1},o_2^{-1},\ldots ,
o_{2k}^{-1})$ and $\mathrm{ (\ref{conknot2})}$ for $(o_1,o_2,\ldots
,o_{2k})$. \label{npretzelconway}
\end{thm}
\begin{proof}
We will only prove (\ref{conknot4}) but one can prove the other by a
similar argument. In the computation tree, we use skein relation
at crossings until vertices of valence 1
in the computation tree up to this point will be $(c_1,c_2,\dots ,c_n)$ where
$c_i$ is either 0 or $\pm 1$. Since the Conway polynomials of
split links vanish, we
may assume there are no than one 0's. The first
term in the parenthesis comes from the case where all $|c_i|$ are 1
because it is again the $(2,\alpha)$ torus link horizontally. It is
a two component link with linking number $- \alpha /2$, so its
Conway polynomial is $-(\alpha /2)z$. For the case where only one
$c_i = 0$, the values on edges to the vertex will contribute exactly
$\cp _{o_i '}$ and the vertex is the composite link of $(2,o_j)$
torus knots $j= 2, \ldots, n$ except $i$.
\end{proof}

\subsection{Conway polynomials of $n$-pretzel links} \label{conoflink}

Since we have already handled links of all odd crossings, we assume that
$n$-pretzel links have at least one even crossing box.
Let $L(p_1,p_2, \ldots, p_n)$ be an $n$-pretzel link and let $s$ be
the number of even $p_i$'s. Then it is a link of $s$ components. The Conway polynomial
$\cp_L$ depends on the choice of the orientation of $L$. There are
$2^{s-1}$ possible orientations of $L$. But one can easily see that
the link always moves in the same direction in all boxes of odd crossings
for arbitrary orientation. For further purpose, we will calculate
the Conway polynomial of the pretzel link with the following
orientations. For the existence of such orientations, we
will prove it in Lemma~\ref{orienlem} : if $n-s$ is even, then there exists an
orientation $O$ of $L$ such that the link $L$ moves in the opposite directions
in all boxes of even $p_i$. If $n-s$ is odd, then there exists an
orientation $O$ of $L$ such that the link $L$ moves in the opposite directions
in all boxes of even $p_i$ except one $p_t$ but without loss of
a generality we assume that $p_1=p_t$.

\begin{thm}
Let $L(p_1,p_2,\ldots ,p_n)$ be a pretzel link with the above
orientation $O$. Let $p_{e_i}=2l_i$ be all even and $p_{o_j}=2k_j+1$ be
all odd. Let $s$ be the number of even $p_i$'s and let $\alpha =
\sum_{i=1}^{n-s}sign(p_{o_i})$ and $\beta = sign(p_1)$. Let
$p_i'=sign(p_i)(|p_i|-1)$. If $n-s$ is even, then the Conway
polynomial of $L(p_1,p_2,\ldots ,p_n)$ is

$$[\prod_{i=1}^{s}(-l_i)]z^{s}(\prod_{i=1}^{n-s}\cp_{p_{o_i}})
[-\frac{\alpha}{2}z+\sum_{i=1}^{n-s}\frac{\cp_{p_{o_i}'}}{\cp_{p_{o_i}}}]
+ [\sum_{i=1}^{s} \prod_{j=1,j\neq i}^{s} (-l_j)]z^{s-1}.$$

If $n-s$ is odd, then the Conway polynomial of $L(p_1,p_2,\ldots
,p_n)$ is
$$[\prod_{i=2}^{s}(-l_i)]z^{s-1}(\prod_{i=1}^{n-s}\cp_{p_{o_i}})\cp_{p_1}
[ -\frac{\alpha+\beta}{2}z+ \frac{\cp_{p_{1}'}}{\cp_{p_1}}
+\sum_{i=1}^{n-s}\frac{\cp_{p_{o_i}'}}{\cp_{p_{o_i}}}]
+[\sum_{i=2}^{s} \prod_{j=2,j\neq i}^{s} (-l_j)]z^{s-2}.$$
\label{conwayoflink}
\end{thm}
\begin{proof}
It is clear by choosing $(p_{e_1}^{-1}$, $p_{e_2}^{-1}$, $\ldots$,
$p_{e_s}^{-1}$, $p_{o_1}$, $\ldots$, $p_{o_{n-s}})$ and $(p_{e_1}, $
$p_{e_2}^{-1},$ $\ldots$, $p_{e_s}^{-1}$, $p_{o_1}$, $\ldots$,
$p_{o_{n-s}})$, respectively.
\end{proof}

More generally, we get the following results by taking
$(p_{e_1}^{-1}$, $p_{e_2}^{-1}$, $\ldots$, $p_{e_t}^{-1}$,
$p_{e_{t+1}}$, $\ldots$, $ p_{e_s}$, $p_{o_1}$, $\ldots$, $
p_{o_{n-s}})$ for a representative of $L(p_1,p_2,\ldots ,p_n)$
induced by an orientation $O$.

\begin{thm}
Let $p_{e_i}=2l_i$ be all even and $p_{o_j}=2k_j+1$ be all odd. Let
$s$ be the number of even $p_i$. Let $t$ be the number of even $p_i$
in the corresponding boxes in which the link moves in the opposite direction, say
$p_{e_i}$ where $i=1, 2$, $\ldots$, $t$. and let $\alpha =
\sum_{j=1}^{n-s}sign(p_{o_j})$ and $\beta =
\sum_{i=t+1}^{s}sign(p_{e_i})$. Let $p_i'=sign(p_i)(|p_i|-1)$. Then
the Conway polynomial of $L(p_1,p_2,\ldots ,p_n)$ with the
orientation $O$ is

\begin{align*}
[\prod_{i=1}^{t}(-l_i)]z^{t}(\prod_{i=1}^{n-s}\cp_{p_{o_i}})
(\prod_{j=1}^{t}\cp_{p_{e_j}}) [ -\frac{\alpha+\beta}{2}z +
\sum_{i=t+1}^{s}\frac{\cp_{p_{e_i}'}}{\cp_{p_{e_i}}} \\
+\sum_{j=1}^{n-s}\frac{\cp_{p_{o_j}'}}{\cp_{p_{o_j}}}]
+[\sum_{i=1}^{t} \prod_{j=1,j\neq i}^{t} (-l_j)]z^{t-1}.
\end{align*} \label{conwayoflink2}
\end{thm}

\section{Genera of $n$-pretzel links}
\label{genera}

We will consider the genus of an $n$-pretzel link with at least one
even crossing box. Let $F_L$ be a Seifert surface of an $n$-pretzel
link $L$. For the rest of the section, let $\chi (\F_L)$ be the
Euler characteristic of $\F_L$, $V$ be the number of Seifert
circles, $E$ be the number of crossings and $F$ be the number of the
components of $L$.

\subsection{Genera of $n$-pretzel knots with one even $p_i$}

We divide into two cases : $i)$ $n$ is odd, $ii)$ $n$ is even. For the
first case: $n$ is odd,
 we can see that
the degree of $\cp_{K(e_1,o_1,o_2,\ldots ,o_n)}$ is
$$2+\prod_{i=2}^{n} \mathrm{degree}(\cp_{o_i})
=2+\sum_{i=2}^{n}(|o_{i}|-1),$$ and the coefficient of this leading
term is $-l\alpha /2$ from Theorem~\ref{npretzelconway}.

Suppose $\alpha$ is nonzero. Then the Seifert surface $\F$ obtained
by applying Seifert's algorithm to the diagram in Figure~\ref{allnpretzel} is a
minimal genus surface. The genus of the Seifert surface $\F_K$ is

\begin{align*}
g(\F_K) &=\frac12\pol [2-\chi (\F_K)] = \frac12~ (2-V+E-F)\\\notag
&=\frac12~ [2-(|e_1|+n-2)+ (|e_1|+\sum_{i=2}^{n}|o_i|) -1]
=\frac12~[2+\sum_{i=2}^{n} (|o_1|-1)]\\\notag & =\frac12~
\mathrm{degree}~ \cp _{K(e_1,o_1,o_2,\ldots ,o_n)}.
\end{align*}

Suppose $\alpha = 0$. This means that we have the same number of
positive and negative twists on odd twists. If we look at the Conway
polynomial in equation~\ref{conknot4}, we drop exactly one in degree with new leading
coefficient $1$. It is sufficient to show that the degree of the
following term is negative. Remark that $\cp _{o_i}=\cp _{-o_i}$.

\begin{align*}
-lz[-\frac{\alpha}{2} z + \sum_{i=2}^{n} \frac{\cp _{o_{i}'}}{\cp
_{o_{i}}}] &= -l[0 + \sum_{i=2}^{n} \frac{z\cp _{o_{i}'}}{\cp
_{o_{i}}}] = -l[ \sum_{i=2}^{n} \frac{sign(o_i)(\cp _{|o_i|}- \cp
_{|o_{i}|-2})}{\cp _{|o_{i}|}}]\\\notag &= -l[ \sum_{i=2}^{n}
(sign(o_i) +  \frac{\cp _{|o_{i}|-2}}{\cp _{|o_i|}})] = -l[
\sum_{i=2}^{n} \frac{\cp _{|o_{i}|-2}}{\cp _{|o_{i}|}}].
\end{align*}

We hope to find a minimal surface of this genus. For the first case,
the sign of an $n$-pretzel is $(\pm ,\pm ,\ldots, \pm
,even,\mp,\mp,\ldots ,\mp)$. The rule is to use the move from the
outmost pair. Then the moves in Figure~\ref{minsurnpretzel} will
increase $V$ by two but will not change $E, F(=1)$; thus we get a
surface with one less genus. If we represent the move by the Conway
notation for algebraic links~\cite{Conway:enum}, it is either
$(\ldots, -a,\ldots,b,\ldots)$ $\Rightarrow$ $(\ldots,
(-1)(-a+1),\ldots,(b-1)(1),\ldots)$ or $(\ldots,
a,\ldots,-b,\ldots)$ $\Rightarrow$ $(\ldots, (1)(a-1),\ldots$,$(-1)$
$(-b+1)$,$\ldots)$ where the sign sum of the $o_i$'s between $a,b$
has to be zero.

For the general case, if we only look at the signs of the odd twists
from $o_1$, we can find a pair $o_i, o_j$ such that we can apply the
move we described above. The resulting diagram satisfies the same hypothesis
with strictly smaller twisted bands. Inductively we get a
well-defined sequence of moves which makes the desired diagram on
which Seifert's algorithm will produce a minimal genus surface.
Figure~\ref{minsurnpretzel} shows the effect on $V, E$. This
completes the case $i)$.

For the second case, $n$ is even, we can see that the degree of
$\cp_{K(e_1,o_1,o_2,\ldots ,o_n)}$ is
$$1+\mathrm{degree}(\cp_{e_1})+\prod_{i=2}^{n}
\mathrm{degree}(\cp_{o_i})=|e_1|+\sum_{i=2}^{n}(|o_{i}|-1),$$ and
the coefficient of the leading term is $-sign(e_1)(\alpha +\beta)/2$
from Theorem~\ref{conwayoflink2}.

Suppose $\alpha +\beta$ is nonzero. Then the Seifert surface $F$ obtained by applying
Seifert's algorithm to the diagram in Figure~\ref{allnpretzel} is a minimal genus surface.
The genus of the Seifert surface $F_K$ is

\begin{align*}
g(F_K) &=\frac12~[2-\chi (F_K)] = \frac12~ (2-V+E-F)\\\notag
     &=\frac12~ [2-(n)+ [|e_1|+\sum_{i=2}^{n}(|o_i|)] -1]
     =\frac12~ [|e_1|+\sum_{i=2}^{n} (|o_1|-1)]\\\notag
     &=\frac12~ \mathrm{degree} \cp _{K(e_1,o_1,o_2,\ldots ,o_n)} .
\end{align*}

Suppose $\alpha +\beta = 0$. This means that we have the same number
of positive and negative twists. As we did before we drop exactly
one in the degree of the Conway polynomial in equation~\ref{conknot5} with new leading
coefficient $1$. All arguments are the same if we change the term in
parentheses in the equation as follows.

\begin{align*}
[\cp_{e_1'} + \cp_{e_1}(-\frac{\beta +\alpha}{2}z + \sum_{i=2}^{n}
\frac{\cp _{o_{i}'}}{\cp _{o_{i}}})] &=\cp_{e_1}[-\frac{\beta
+\alpha}{2}z + \frac{\cp _{e_{1}'}}{\cp _{e_{1}}} + \sum_{i=2}^{n}
\frac{\cp _{o_{i}'}}{\cp _{o_{i}}}].
\end{align*}

\begin{figure}
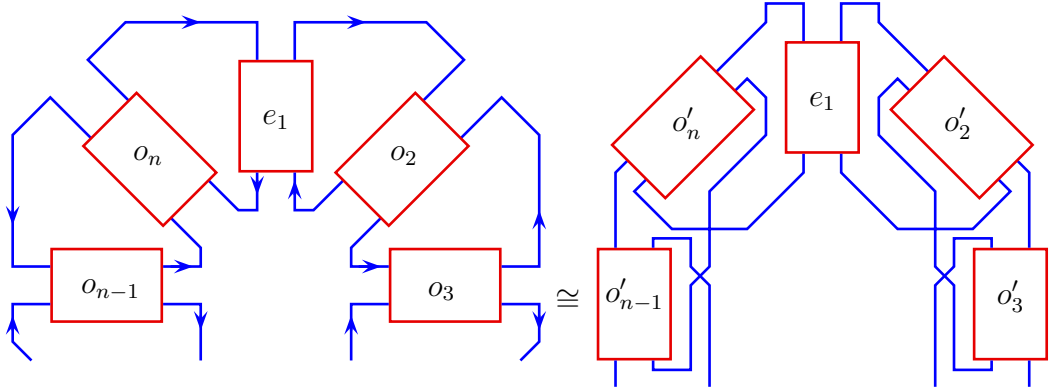
 $$\pspicture[.2](-3.5,-1.4)(3.6,4)
\rput(0,2.2){\rnode{b1}{$e_1$}}
\psframe[linecolor=darkred](-.5,1.5)(.5,3)
\psframe[linecolor=darkred](1.5,-.5)(3,.5)
\psframe[linecolor=darkred](-1.5,-.5)(-3,.5)
\psline(1,1.5)(.5,1)(.25,1)(.25,1.5)
\psline[arrowscale=1.5]{->}(.25,1.2)(.25,1.4)
\psline(-1,1.5)(-.5,1)(-.25,1)(-.25,1.5)
\psline[arrowscale=1.5]{<-}(-.25,1.2)(-.25,1.4)
\psline(1.5,1)(1,.5)(1,.25)(1.5,.25)
\psline[arrowscale=1.5]{->}(1.1,.25)(1.4,.25)
\psline(-1.5,1)(-1,.5)(-1,.25)(-1.5,.25)
\psline[arrowscale=1.5]{->}(-1.4,.25)(-1.1,.25)
\psline(1.5,-.25)(1,-.25)(1,-1)
\psline[arrowscale=1.5]{->}(1,-.6)(1,-.4)
\psline(-1.5,-.25)(-1,-.25)(-1,-1)
\psline[arrowscale=1.5]{<-}(-1,-.6)(-1,-.4)
\psline(2,1.5)(3,2.5)(3.5,2)(3.5,.25)(3,.25)
\psline[arrowscale=1.5]{<-}(3.5,1)(3.5,.8)
\psline(-2,1.5)(-3,2.5)(-3.5,2)(-3.5,.25)(-3,.25)
\psline[arrowscale=1.5]{->}(-3.5,1)(-3.5,.8)
\psline(1.5,2)(2.5,3)(2,3.5)(.25,3.5)(.25,3)
\psline[arrowscale=1.5]{<-}(1.2,3.5)(1,3.5)
\psline(-1.5,2)(-2.5,3)(-2,3.5)(-.25,3.5)(-.25,3)
\psline[arrowscale=1.5]{->}(-1.2,3.5)(-1,3.5)
\psline(3,-.25)(3.5,-.25)(3.5,-.75)(3.25,-1)
\psline[arrowscale=1.5]{->}(3.5,-.4)(3.5,-.6)
\psline(-3,-.25)(-3.5,-.25)(-3.5,-.75)(-3.25,-1)
\psline[arrowscale=1.5]{<-}(-3.5,-.4)(-3.5,-.6)
\pspolygon[linecolor=darkred,fillcolor=white,fillstyle=solid](1.5,.793)(.793,1.5)(1.853,2.56)(2.56,1.853)
\pspolygon[linecolor=darkred,fillcolor=white,fillstyle=solid](-1.5,.793)(-.793,1.5)(-1.853,2.56)(-2.56,1.853)
 \rput[b](1.7,1.6){$o_2$}  \rput[b](-1.7,1.6){$o_{n}$}
\rput[b](2.2,-.2){$o_3$}  \rput[b](-2.2,-.2){$o_{n-1}$}
\endpspicture \cong \pspicture[.2](-3.1,-1.65)(3.5,3.75)
\rput(0,2.2){\rnode{b1}{$e_1$}}
\psframe[linecolor=darkred](-.5,1.5)(.5,3)
\psframe[linecolor=darkred](2,-1.25)(3,.25)
\psframe[linecolor=darkred](-2,-1.25)(-3,.25)
\psline(.25,3)(.25,3.5)(.75,3.5)(.75,3.25)(2.75,1.25)(2.75,.25)
\psline(2.75,-1.25)(2.75,-1.6)
\psline(-.25,3)(-.25,3.5)(-.75,3.5)(-.75,3.25)(-2.75,1.25)(-2.75,.25)
\psline(-2.75,-1.25)(-2.75,-1.6)
\psline(2.25,-1.25)(2.25,-1.375)(1.75,-1.375)(1.75,-.25)(1.5,0)(1.5,1)(.75,1.75)(.75,2.25)(1,2.5)(2.5,1)(2,.5)(1,.5)(.25,1.25)(.25,1.5)
\psline(2.25,.25)(2.25,.375)(1.75,.375)(1.75,0)(1.5,-.25)(1.5,-1.6)
\psline(-2.25,-1.25)(-2.25,-1.375)(-1.75,-1.375)(-1.75,-.25)(-1.5,0)(-1.5,1)(-.75,1.75)(-.75,2.25)(-1,2.5)(-2.5,1)(-2,.5)(-1,.5)(-.25,1.25)(-.25,1.5)
\psline(-2.25,.25)(-2.25,.375)(-1.75,.375)(-1.75,0)(-1.5,-.25)(-1.5,-1.6)
\pspolygon[linecolor=darkred,fillcolor=white,fillstyle=solid](1.618,2.795)(2.795,1.618)(2.088,.911)(.911,2.088)
\pspolygon[linecolor=darkred,fillcolor=white,fillstyle=solid](-1.618,2.795)(-2.795,1.618)(-2.088,.911)(-.911,2.088)
 \rput[b](1.8,1.7){$o_2'$}  \rput[b](-1.8,1.7){$o_{n}'$}
\rput[b](2.5,-.6){$o_3'$}  \rput[b](-2.5,-.6){$o_{n-1}'$}
\endpspicture $$
\caption{How to modify a diagram in Figure~\ref{allnpretzel} to find a minimal genus diagram of
$L(p_1, p_2, \ldots , p_n)$.} \label{minsurnpretzel}
\end{figure}

We can find a minimal surface of this genus by the same method as
shown in Figure~\ref{minsurnpretzel} if we handle the even crossing
box together. This gives us the following theorem.

\begin{thm}
Let $K(p_1,o_2,o_3, \ldots, o_n)$ be an $n$-pretzel knot with one
even $p_1$. Let $\alpha$ $= \sum_{i=2}^{n}$ $ sign(o_i)$ and $\beta
$$= sign(p_1)$. Suppose $|p_1|, |o_i| \ge 2$. Let
$$\delta=\sum_{i=2}^{n}(|o_{i}|-1).$$ Then the
genus $g(K)$ of $K$ is
$$g(K) = \left\{
\begin{array}{cl}
\frac12~ (\delta+2) & ~~\mathrm{if}~ n~
\mathrm{is}~\mathrm{odd}~\mathrm{and}~ \alpha \neq 0, \\
 \frac12~ \delta & ~~\mathrm{if}~ n~
\mathrm{is}~\mathrm{even}~\mathrm{and}~
 \alpha = 0, \\
 \frac12~ (|p_1|+\delta) & ~~\mathrm{if}~ n
~\mathrm{is}~\mathrm{even}~\mathrm{and}~\alpha + \beta \neq 0, \\
\frac12~ (|p_1|+\delta)-1
  & ~~\mathrm{if}~ n ~\mathrm{is}~\mathrm{even}~\mathrm{and}~
 \alpha + \beta = 0.
\end{array} \right.
$$
 \label{knotgenus}
\end{thm}

\subsection{Genera of $n$-Pretzel links}

Intuitively, if we have more even $p_i$'s with opposite directions,
then we will have a surface of smaller genus. So we want to choose an orientation
which forces all the even $p_i$'s to move in the opposite directions,
but this may not be possible for all cases.

\begin{figure}
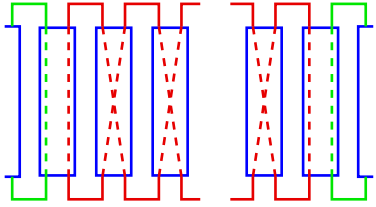
 $$\pspicture[.2](-3.5,-1.5)(3.5,1.5)
\psframe(-.25,-1)(.25,1) \psframe(-1,-1)(-.5,1)
\psframe(-1.75,-1)(-1.25,1) \psline(-2.2,1)(-2,1)(-2,-1)(-2.2,-1)
\psframe(1,-1)(1.5,1) \psframe(1.75,-1)(2.25,1)
\psline(2.7,1)(2.5,1)(2.5,-1)(2.7,-1)
\psline[linecolor=emgreen](-2.1,1)(-2.1,1.3)(-1.65,1.3)(-1.65,1)
\psline[linecolor=emgreen,linestyle=dashed](-1.65,1)(-1.65,-1)
\psline[linecolor=darkred,linestyle=dashed](-1.35,1)(-1.35,-1)
\psline[linecolor=emgreen](-2.1,-1)(-2.1,-1.3)(-1.65,-1.3)(-1.65,-1)
\psline[linecolor=darkred](-1.35,1)(-1.35,1.3)(-.9,1.3)(-.9,1)
\psline[linecolor=darkred,linestyle=dashed](-.9,1)(-.6,-1)
\psline[linecolor=darkred,linestyle=dashed](-.9,-1)(-.6,1)
\psline[linecolor=darkred](-1.35,-1)(-1.35,-1.3)(-.9,-1.3)(-.9,-1)
\psline[linecolor=darkred](-.6,1)(-.6,1.3)(-.15,1.3)(-.15,1)
\psline[linecolor=darkred,linestyle=dashed](-.15,1)(.15,-1)
\psline[linecolor=darkred,linestyle=dashed](-.15,-1)(.15,1)
\psline[linecolor=darkred](-.6,-1)(-.6,-1.3)(-.15,-1.3)(-.15,-1)
\psline[linecolor=darkred](.4,1.3)(.15,1.3)(.15,1)
\psline[linecolor=darkred](.4,-1.3)(.15,-1.3)(.15,-1)
\psline[linecolor=darkred](.8,1.3)(1.1,1.3)(1.1,1)
\psline[linecolor=darkred,linestyle=dashed](1.1,1)(1.4,-1)
\psline[linecolor=darkred,linestyle=dashed](1.1,-1)(1.4,1)
\psline[linecolor=darkred](.8,-1.3)(1.1,-1.3)(1.1,-1)
\psline[linecolor=darkred](1.4,1)(1.4,1.3)(1.85,1.3)(1.85,1)
\psline[linecolor=darkred,linestyle=dashed](1.85,1)(1.85,-1)
\psline[linecolor=darkred](1.4,-1)(1.4,-1.3)(1.85,-1.3)(1.85,-1)
\psline[linecolor=emgreen](2.15,1)(2.15,1.3)(2.6,1.3)(2.6,1)
\psline[linecolor=emgreen,linestyle=dashed](2.15,1)(2.15,-1)
\psline[linecolor=emgreen](2.15,-1)(2.15,-1.3)(2.6,-1.3)(2.6,-1)
\endpspicture $$
\caption{Boundary orientation of $L(p_1, p_2, \ldots , p_n)$.}
\label{orienpretlink}
\end{figure}

\begin{lem}
Let $L(p_1,p_2, \ldots, p_n)$ be an $n$-pretzel link and let $s$ be
the number of even $p_i$'s. If $n-s$ is even, then there exists an
orientation of $L$ such that the link $L$ moves in opposite
directions in all boxes of even $p_i$. If $n-s$ is odd and a given
$p_t$ is  even, then there exists an orientation of $L$ such that
the link $L$ moves in opposite directions in all boxes of even
$p_i$'s  except the one corresponding to $p_t$.
\label{orienlem}\end{lem}
\begin{proof}
If all $p_{j}$ between two even $p_i$ and $p_k$ are odd, the number
of these $p_j$'s odd $(mod~2)$ will decide the boundary orientation
as depicted in Figure~\ref{orienpretlink}.

If the number of odd crossing boxes is even,
we can orient the link such that the link moves
oppositely in all boxes of even crossings.
Otherwise there is just one box for which the link moves
in the same direction.
So starting from $p_t$ will complete the proof.
\end{proof}

Let us denote the orientation we choose in Lemma~\ref{orienlem} by $O'$.
From Theorem~\ref{conwayoflink}, we can do almost the same
comparison by using equation (\ref{genineqal}). But we have to be
careful to use (\ref{genineqal}) for links. Since it was defined for
oriented links, we can interpolate it as follows.

\begin{align*}
g(L)= ~\mathrm{min}_{O}\{\mathrm{min} & \{\mathrm{genus}~
\mathrm{of}~ \F_{(L,O)}~ |~ \F_{L,O}~\mathrm{is}~
\mathrm{a}~\mathrm{Seifert}~\mathrm{surface}
\\\notag & \mathrm{with}~ \mathrm{the}~ \mathrm{orientation}~ O\}\}.
\end{align*}
where the first $O$ runs over all possible orientations of $L$. So
(\ref{genineqal}) gives us an inequality on the second minimum of the
fixed orientation $O$ and $\cp_{(L,O)}$.

We divide into two cases : $i)$ $n-s$ is even, $ii)$ $n-s$ is
odd. For the first case, $n-s$ even, we can see that the degree of
$\cp_{L(p_1,p_2,\ldots ,p_n)}$ is $$s+\prod_{i=1}^{n-s}
\mathrm{degree}(\cp_{p_{m_i}})+1 = s+\sum_{i=1}^{n-s}(|p_{m_i}|-1)
+1,$$ and the coefficient of this leading term is $-\alpha /2$ from
Theorem~\ref{conwayoflink}.

Suppose $\alpha$ is nonzero. Then the Seifert surface $\F$ obtained
by applying Seifert's algorithm with the fixed orientation $O'$ is
a minimal genus surface of $(L,O')$. Let us find the genus of
the Seifert surface $\F_{(L,O')}$.

\begin{align*}
2g(\F_L) &=2-\chi (\F_L) = 2-(V-E+F)\\\notag
     &=2-(n-s)+(\sum_{i=1}^{n-s} (|p_{m_j}|-1)) + [\sum_{j=1}^{s}|p_{i_{j}}| +
          \sum_{i=1}^{n-s}(|p_{m_{i}}|)] - s\\\notag
     &=2 + \sum_{i=1}^{n-s} (|p_{k_i}|-1)
     =\mathrm{degree} (\cp_L)-s+1.
\end{align*}

For the rest of the cases of the arguments are parallel to the argument for knots.
Next, we explain how $p_t$ will be chosen for the rest of the article.
\begin{rem}
First, we look at the minimum of the absolute value of $p_{e_i}$
over all even crossings. If the minimum is taken by the unique $p_{e_i}$ or by $p_{e_i}$'s of the
same sign, we choose it for $p_t$. If there are more than two $p_{e_i}$'s with
different signs and the same absolute value, then we look at the value $\alpha$, the sign sum of
odd crossings. If it is neither $1$ nor $-1$, then we pick the positive
one for $p_t$. If $\alpha =1(-1)$, pick the negative(positive) one
for $p_t$. \label{defpt}
\end{rem}

For the second case, $n-s$ odd, we find $p_t$ as described the above.
For the last two cases, we will drop the genus by $1$.
Denote the orientation we chose here by $O_1$.

\begin{lem}
For an arbitrary orientation  $O$, $\mathrm{degree} \cp_{(L,O)} \ge
\mathrm{degree} \cp_{(L,O_1)}$. \label{orengenuslem}\end{lem}
\begin{proof}
If we count $t_O$, the number of even crossings in which the link
moves in the opposite directions with respect to $O$, we can see that
$t_O \le t_{O_{1}}$. If we look at the Conway polynomial in
Theorem~\ref{conwayoflink2}, we have that $i)$ we can ignore the
second term, $ii)$ increasing $t$ by $1$ will change the degree of the
second term by $-(|p_i|-2)$, and by hypothesis, $|p_i|\ge 2$.
\end{proof}

\begin{thm}
Let $L(p_1,o_2, \ldots, o_s, e_{s+1}, \ldots, e_n)$ be an
$n$-pretzel link with at least one even $p_i$. Let $\alpha $ $=
\sum_{i=2}^{n-s}$ $ sign(p_{o_i})$ and $\beta = sign(p_t)$. Suppose
$|o_i|,|e_j| \ge 2$. Let $p_t$ be the integer described in
Remark~\ref{defpt}. Let $l$ be the number of even $p_i$'s. Let
$$\delta=\sum_{i=2}^{n-s}(|o_{i}|-1).$$ Then the genus $g(L)$ of $L$ is

$$g(L) = \left\{
\begin{array}{cl}
\frac12~\delta +1 & ~~\mathrm{if}~ n-s~
\mathrm{is}~\mathrm{even}~\mathrm{and}~ \alpha \neq 0, \\
 \frac12\delta & ~~\mathrm{if}~ n-s~
\mathrm{is}~\mathrm{even}~\mathrm{and}~
 \alpha = 0, \\
 \frac12(|p_t|+\delta) & ~~\mathrm{if}~ n-s
~\mathrm{is}~\mathrm{odd}~\mathrm{and}~\alpha + \beta \neq 0, \\
\frac12(|p_t|+\delta)-1
  & ~~\mathrm{if}~ n-s ~\mathrm{is}~\mathrm{odd}~\mathrm{and}~
 \alpha + \beta = 0.
\end{array} \right.
$$
 \label{linkgenus}
\end{thm}
\begin{proof}
It follows from Theorem~\ref{conwayoflink}, \ref{conwayoflink2} and Lemma~\ref{orengenuslem}.
\end{proof}

\section{The basket numbers of pretzel links} \label{appl}

First let us recall a definition of the basket number.
Let $A_n \subset \mathbb{S}^3$ denote an $n$-twisted unknotted
annulus.
A Seifert surface $\F$ is a basket surface if $\F =
D_2$ or if $\F = \F_0
*_{\alpha} A_n$ which can be constructed by plumbing $A_n$ to a
basket surface $\F_0$ along a proper arc $\alpha \subset D_2 \subset \F_0$. A basket number of a link $L$,
denoted by $bk(L)$, is the minimal number of annuli used to obtain a
basket surface $\F$ such that $\partial \F=L$. For standard
definitions and notations, we refer to \cite{Rudolph:plumbing}.
Throughout the section, we will assume all links are not splitable,
$i. e.$, Seifert surfaces are connected. Otherwise, we can handle
each connected component separately.

For the basket number and the genus of a link, there is a useful
theorem.

\begin{thm} [\cite{BKP}]
Let $L$ be a link of $l$ components. Then
the basket number of $L$ is bounded by the genus and the canonical genus of $L$ as,
$$ 2g(L)+ l-1  \le
bk(L) \le 2g_c(L) +l-1.$$ \label{bkgenusthm1}
\end{thm}

Since we have found that a minimal genus surface of a pretzel link
$L$ of genus $g(L)$ can be obtained by applying Seifert algorithm on
a diagram of $L$, $i. e.$, $g(L)=g_c(L)$, we find that the basket
number of a pretzel link $L$ is $2g(L)+ l-1$, $i. e.$, $bk(L)=2g(L)+
l-1$.

\begin{thm}
Let $K(p_1,o_2,o_3, \ldots, o_n)$ be an $n$-pretzel knot with one
even $p_1$. Let $\alpha$ $= \sum_{i=2}^{n}$ $ sign(o_i)$ and $\beta
$$= sign(p_1)$. Suppose $|p_1|, |o_i| \ge 2$. Let
$$\delta=\sum_{i=2}^{n}(|o_{i}|-1).$$
Then the basket number $bk(K)$ of $K$,
$$bk(K) = \left\{
\begin{array}{cl}
\delta+2  & ~~\mathrm{if}~ n~
\mathrm{is}~\mathrm{odd}~\mathrm{and}~ \alpha \neq 0, \\
 \delta & ~~\mathrm{if}~ n~
\mathrm{is}~\mathrm{even}~\mathrm{and}~
 \alpha = 0, \\
 |p_1|+\delta & ~~\mathrm{if}~ n
~\mathrm{is}~\mathrm{even}~\mathrm{and}~\alpha + \beta \neq 0, \\
|p_1|+\delta-2
  & ~~\mathrm{if}~ n ~\mathrm{is}~\mathrm{even}~\mathrm{and}~
 \alpha + \beta = 0.
\end{array} \right.
$$
 \label{knotbk}
\end{thm}

\begin{thm}
Let $L(p_1,o_2, \ldots, o_s, e_{s+1}, \ldots, e_n)$ be an
$n$-pretzel link with at least one even $p_i$. Let $\alpha $ $=
\sum_{i=2}^{n-s}$ $ sign(p_{o_i})$ and $\beta = sign(p_t)$. Suppose
$|o_i|,|e_j| \ge 2$. Let $p_t$ be the integer described in
Remark~\ref{defpt}. Let $l$ be the number of even $p_i$'s. Let
$$\delta=\sum_{i=2}^{n-s}(|o_{i}|-1).$$ Then the basket number
$bk(L)$ of $L$,
$$bk(L) = \left\{
\begin{array}{cl}
\delta +l+1  & ~~\mathrm{if}~ n-s~
\mathrm{is}~\mathrm{even}~\mathrm{and}~ \alpha \neq 0, \\
\delta +l-1 & ~~\mathrm{if}~ n-s~
\mathrm{is}~\mathrm{even}~\mathrm{and}~
 \alpha = 0, \\
 |p_t|+\delta +l-1 & ~~\mathrm{if}~ n-s
~\mathrm{is}~\mathrm{odd}~\mathrm{and}~\alpha + \beta \neq 0, \\
|p_t|+\delta +l-3
  & ~~\mathrm{if}~ n-s ~\mathrm{is}~\mathrm{odd}~\mathrm{and}~
 \alpha + \beta = 0.
\end{array} \right.
$$
\label{linkbk}
\end{thm}

\section*{Acknowledgments}
The author would like to thank Cameron Gordon for helpful
discussion, valuable comments on this work. The
\TeX\, macro package PSTricks~\cite{PSTricks} was essential for
typesetting the equations and figures.

\end{document}